\documentclass[10pt]{article}
\usepackage{graphicx,color} 
\usepackage{mathtools,amsmath}
 \usepackage{tikz}
\usepackage[utf8]{inputenc}
\usepackage[T1]{fontenc}
\usepackage{amssymb}
\usepackage{graphicx}
\usepackage{hyperref}
\usepackage{amsmath,a4wide}
\usepackage{amssymb}
\usepackage{amsthm}
\author{Matthieu Calvez\footnote{Universidad de Valpara\'iso, IMUV, Gran Breta\~na 1111, Valpara\'iso 2340000, Chile}, Bruno A. Cisneros de la Cruz\footnote{CONACYT - Instituto de Matem\'aticas de la UNAM, Oaxaca, Le\'on No.2, Oaxaca de Ju\'arez, Mexico}}
\title{Curve graphs for Artin-Tits groups of type $B$, $\widetilde A$ and $\widetilde C$ are hyperbolic}
\newtheorem{definition}{Definition}[section]
\newtheorem{theorem}[definition]{Theorem}

\newtheorem{corollary}[definition]{Corollary}
\newtheorem{lemma}[definition]{Lemma}

\newtheorem{remark}[definition]{Remark}
\newtheorem{proposition}[definition]{Proposition}

\newtheorem{problem}[definition]{Problem}

\newtheorem{Proposition-Definition}{Proposition-Definition}

\newtheorem*{theorem*}{Theorem}

\numberwithin{equation}{section}
\begin{document}

\maketitle
\begin{abstract}
The \emph{graph of irreducible parabolic subgroups} is a combinatorial object
associated to an Artin-Tits group $A$ defined so as to coincide with the curve graph of the $(n+1)$-times punctured disk when $A$ is Artin's braid group on $(n+1)$ strands. In this case, it is a hyperbolic graph, by the celebrated Masur-Minsky's theorem. 
Hyperbolicity for more general Artin-Tits groups is an important open question. In this paper, we give a partial affirmative answer. 

For $n\geqslant 3$, we prove that the graph of irreducible parabolic subgroups associated to the Artin-Tits group of spherical type $B_n$ is also isomorphic to the curve graph of the $(n+1)$-times punctured disk, hence it is hyperbolic.

For $n\geqslant 2$, 
we show that the graphs of irreducible parabolic subgroups associated to the Artin-Tits groups of euclidean type $\widetilde A_n$ and $\widetilde C_n$ are isomorphic to some subgraphs of the curve graph of the $(n+2)$-times punctured disk which are not quasi-isometrically embedded. 
We prove nonetheless that these graphs are hyperbolic. 
\end{abstract}

\section{Introduction  and background}\label{S:Intro}

An \emph{Artin-Tits group} is a group defined by a presentation involving a finite set of generators~$S$ (the \emph{standard generators}) and where all the relations are as follows: every pair $(a,b)$ of standard generators satisfies at most one balanced relation of the form 
$$\Pi(a,b;m_{a,b})=\Pi(b,a;m_{a,b}),$$
with $m_{a,b}\geqslant 2$ and where for $k\geqslant 2$, $$\Pi(a,b;k)=\begin{cases} (ab)^{\frac{k}{2}}  & \text{if $k$ is even},\\
(ab)^{\frac{k-1}{2}}a & \text{if $k$ is odd.}
\end{cases}$$

This presentation can be encoded by a \emph{Coxeter graph} $\Gamma$. The vertices of $\Gamma$ are in bijection with the set~$S$. Two distinct vertices $a,b$ of $\Gamma$ are connected by a labeled edge if and only if either they satisfy no relation, in which case the label is $\infty$, or $m_{a,b}>2$, in which case the label
is $m_{a,b}$.
%labeled $m_{ab}$ if $m_{ab}>2$ and labeled by $\infty$ if~$a,b$ satisfy no relation. 
The Artin-Tits group defined by the Coxeter graph $\Gamma$ will be denoted~$A_{\Gamma}$. Because all the relations in the given presentation of~$A_{\Gamma}$ are balanced, there is a homomorphism $\epsilon_{\Gamma}: A_{\Gamma}\longrightarrow \mathbb Z$ assigning to each element of $A_{\Gamma}$ the exponent sum of any word on~$S$ representing it.

The group $A_{\Gamma}$ is said to be \emph{irreducible} if $\Gamma$ is connected and \emph{dihedral} if~$\Gamma$ has two vertices. The quotient by the normal subgroup generated by the squares of the elements in $S$ is a \emph{Coxeter group} denoted by $W_{\Gamma}$. 
The Artin-Tits group $A_{\Gamma}$ is said to be of \emph{spherical type} if $W_{\Gamma}$ is finite. 

Given a proper subset $\emptyset \neq X\subsetneq S$, the proper subgroup of $A_{\Gamma}$ generated by $X$ is called a \emph{standard parabolic subgroup} of~$A_{\Gamma}$; it is naturally isomorphic to the Artin-Tits group $A_{\Xi}$ defined by the subgraph $\Xi$ of $\Gamma$ induced by the vertices in~$X$ \cite{VanDerLek}. A subgroup $P$ of $A_{\Gamma}$ is called \emph{parabolic} if it is conjugate to a standard parabolic subgroup. 

The flagship example of an Artin-Tits group (of spherical type) is the braid group on $(n+1)$ strands --i.e. the Artin-Tits group defined by the graph $A_n$ shown in Figure \ref{Figure}(a). This group is isomorphic to the Mapping Class Group of an $(n+1)$-times punctured closed disk $\mathbb D_{n+1}$, that is the group of isotopy classes of orientation-preserving homeomorphisms of $\mathbb D_{n+1}$ which induce the identity on the boundary of $\mathbb D_{n+1}$.

The \emph{curve graph} $\mathcal{CG}(\mathbb D_{n+1})$ of the $(n+1)$-times punctured disk $\mathbb D_{n+1}$ is the graph whose vertices are the isotopy classes of \emph{essential} simple closed curves in $\mathbb D_{n+1}$ (closed curves without auto-intersection and enclosing at least 2 and at most $n$ punctures) and where two vertices are joined by an edge if the corresponding isotopy classes of curves are distinct and admit disjoint representatives. This is a connected graph whenever $n\geqslant 3$, which we will suppose henceforth. 
The graph~$\mathcal{CG}(\mathbb D_{n+1})$ is equipped with the combinatorial metric $d_{\mathbb D_{n+1}}$ defined by declaring each edge to have length one. There is a natural action of the Artin-Tits group $A_{A_n}$ on the set of isotopy classes of essential simple closed curves in $\mathbb D_{n+1}$; 
this action preserves adjacency in $\mathcal{CG}(\mathbb D_{n+1})$, hence~$A_{A_n}$ acts
 by isometries on the curve graph of~$\mathbb D_{n+1}$.

Following Masur-Minsky's celebrated theorem~\cite[Theorem 1.1]{MasurMinsky1}, $\mathcal{CG}(\mathbb D_{n+1})$ is a hyperbolic metric space. Furthermore, the Artin-Tits group $A_{A_n}$ (viewed as the mapping class group of $\mathbb D_{n+1}$) 
%fits in the more general framework of what is now called 
is a \emph{hierarchically hyperbolic space} \cite{Sisto, BHS1,BHS2}, with respect to the projections to the curve graphs of subsurfaces of the disk. A natural and challenging question is whether any irreducible Artin-Tits group~$A_{\Gamma}$ (not necessarily of spherical type) admits such a hierarchical structure. A first step forward is to define a hyperbolic space on which $A_{\Gamma}$ acts in the same way as the braid group acts on the curve graph of the punctured disk.

In \cite{CGGMW}, Cumplido, Gebhardt, Gonz\'alez-Meneses and Wiest explain a one-to-one correspondence between isotopy classes of simple closed curves in $\mathbb D_{n+1}$ and proper irreducible parabolic subgroups of $A_{A_n}$ which allows them to translate the definition of the curve graph $\mathcal{CG}(\mathbb D_{n+1})$ in purely algebraic terms. This definition can then be generalized to any Artin-Tits group as follows: 

%
%There is a natural candidate for this purpose, the so-called \emph{graph of irreducible parabolic subgroups} $\mathcal C_{parab}({\Gamma})$, which was recently introduced by Cumplido, Gebhardt, Gonz\'alez-Meneses and Wiest \cite[Definition 2.3]{CGGMW}. The definition can be stated as follows. 

\begin{definition}\cite{CGGMW}\cite[Definition 4.1]{MorrisWright} \label{D:CParab} Let $A_{\Gamma}$ be an Artin-Tits group. 
Two distinct proper irreducible parabolic subgroups $P$ and $Q$ are called \emph{adjacent} if one of the following conditions holds.
\begin{itemize}
\item $P\subset Q$ or $Q\subset P$,
\item $P\cap Q=\{1\}$ and $pq=qp$ for all $p\in P$ and $q\in Q$.
\end{itemize}
The \emph{graph of irreducible parabolic subgroups} of $A_{\Gamma}$ is the graph $\mathcal C_{parab}(\Gamma)$ whose vertices are the proper irreducible parabolic subgroups of $A_{\Gamma}$ and where two vertices 
are connected by an edge if and only if they correspond to adjacent parabolic subgroups. 
\end{definition}

The graph $\mathcal C_{parab}(\Gamma)$ is equipped with a metric, declaring each edge to have length one. We denote by $d_{\Gamma}$ the distance on $\mathcal C_{parab}(\Gamma)$. There is a natural simplicial action of $A_{\Gamma}$ on $\mathcal C_{parab}(\Gamma)$, by conjugation on parabolic subgroups. This action will be denoted on the right: given a proper irreducible parabolic subgroup $P$ of $A_{\Gamma}$ and $g\in A_{\Gamma}$, the parabolic subgroup $g^{-1}Pg$ will be denoted~$P^g$. Accordingly, we will always use the exponent notation for conjugacy in a group $G$: given $g,h\in G$, $h^g=g^{-1}hg$. 

Note that $\mathcal C_{parab}(\Gamma)$ is empty if $\Gamma$ consists of a single vertex ($A_{\Gamma}$ is cyclic). 
If $A_{\Gamma}$ is dihedral, then by \cite[Lemma 5.2, Theorem 5.3]{MorrisWright}, $\mathcal C_{parab}(\Gamma)$ has infinite diameter and is not connected, unless~$\Gamma$ consists of two vertices with no edge, in which case $A_{\Gamma}\cong\mathbb Z^2$ and $\mathcal C_{parab}(\Gamma)$ consists of two vertices and a single edge between them. Also, if $\Gamma$ is not connected,  $\mathcal C_{parab}(\Gamma)$ is easily shown to have diameter 2. In this paper, we will always assume that $\Gamma$ is connected and has at least 3 vertices.

%It is worth noting that the definition of adjacency may be greatly simplified when $A_{\Gamma}$ is of spherical type, or of FC-type -- see \cite{CGGMW} and \cite{MorrisWright}. 

Most of the known properties of the graph $\mathcal C_{parab}(\Gamma)$ are gathered in \cite{CalvezWiest} and \cite{MorrisWright}. For example, if~$A_{\Gamma}$ is irreducible and of spherical type ($\Gamma$ with at least 3 vertices), then~$\mathcal C_{parab}(\Gamma)$ is connected and has infinite diameter (\cite[Lemma 5.2]{MorrisWright} and \cite[Corollary 4.13]{CalvezWiest}). Finally, when $\Gamma=A_n$, $\mathcal C_{parab}(A_n)$ is isomorphic to the curve graph $\mathcal{CG}(\mathbb D_{n+1})$ and it is hyperbolic in virtue of Masur-Minsky's theorem. 
As noted in \cite{MorrisWright}, hyperbolicity of $\mathcal C_{parab}({\Gamma})$ is currently ``the major challenge for research moving forward in the area''. 
%So far, the only known affirmative result concerns Artin's braid groups (type~$A_n$). This can be synthetized as follows: 
%
%\begin{theorem}\cite{CGGMW}\label{T:Intro}
%%\begin{itemize}
%Let $n\geqslant 3$. The graph of irreducible parabolic subgroups $\mathcal C_{parab}({A_n})$ is \emph{isomorphic} to the curve graph of the $(n+1)$-times punctured disk $\mathbb D_{n+1}$. Therefore $\mathcal C_{parab}({A_n})$ is hyperbolic. 
%%\item[(2)] The complex of irreducible parabolic subgroups associated to a non-abelian dihedral Artin-Tits group is infinite and totally disconnected, hence it is not hyperbolic (the complex of irreducible parabolic subgroups of the abelian dihedral Artin-Tits group consists of two vertices connected by an edge, hence it is hyperbolic). 
%%\end{itemize}
%\end{theorem}

\begin{figure}
\center
\includegraphics{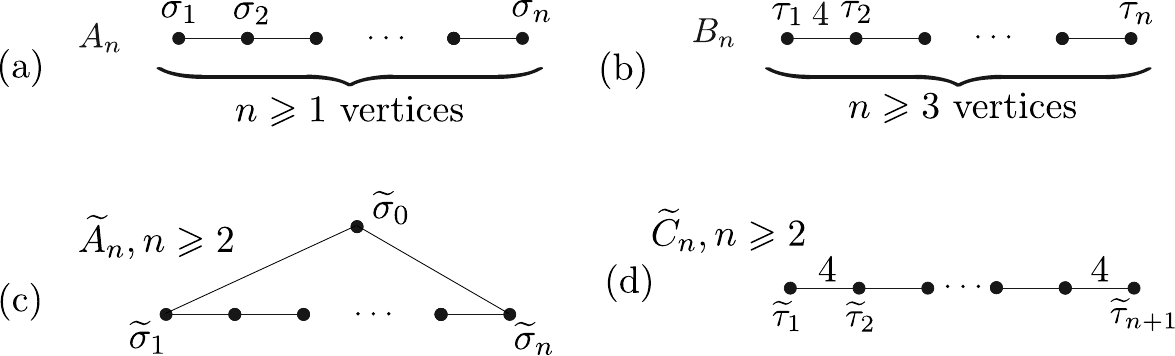}
\caption{The four infinite families of Coxeter graphs involved in the paper. As usual, we omit the label $m_{ab}$ whenever $m_{ab}=3$. The given labeling of the standard generators will be used throughout the paper. }\label{Figure}
\end{figure}

In this paper, we study the graphs of irreducible parabolic subgroups of three infinite families of Artin-Tits groups closely related to Artin's braid groups, whose defining Coxeter graphs are depicted in Figure~\ref{Figure} (b)-(c)-(d). The group $A_{B_n}$ is of spherical type while $A_{\widetilde A_n}$ and $A_{\widetilde C_n}$ are of \emph{euclidean} type. Whereas Artin-Tits groups of spherical type were long well-known (\cite{BS}, for instance), the structure of euclidean Artin-Tits groups was elucidated only recently by McCammond and his collaborators \cite{McCS}. 

Here is a brief summary of the results in the paper. 
Firstly, we recall that the Artin-Tits group~$A_{B_n}$ can be realized as a subgroup of index ($n+1$) of Artin's braid group on $(n+1)$ strands~$A_{A_n}$ \cite{Peifer}. We shall prove that this inclusion induces a graph isomorphism between the respective graphs of irreducible parabolic subgroups: 

\begin{theorem}\label{T:B}

For $n\geqslant 3$, the graph $\mathcal C_{parab}(B_n)$ is isomorphic to $\mathcal{CG}(\mathbb D_{n+1})$ --and to $\mathcal C_{parab}(A_n)$. Therefore, $\mathcal C_{parab}(B_n)$ is hyperbolic. 

\end{theorem}

We then focus on the Artin-Tits groups $A_{\widetilde A_n}$ and $A_{\widetilde C_n}$. Our arguments build on two classical embeddings between the Artin-Tits groups involved. First, there is an embedding of $A_{\widetilde A_n}$ in $A_{B_{n+1}}$ which yields a 
semi-direct product decomposition $A_{B_{n+1}}\cong A_{\widetilde A_n}\rtimes \mathbb Z$~\cite{Peifer}.
Second, $A_{\widetilde C_n}$ embedds as a subgroup of index $(n+1)(n+2)$ in Artin's braid group on $(n+2)$ strands $A_{A_{n+1}}$ --in a very similar way 
as $A_{B_n}$ embeds in $A_{A_n}$ \cite{Allcock}.

%This group can be embedded in the Artin-Tits group of type $B_{n+1}$ and $A_{B_{n+1}}$ can be decomposed as a semi-direct product $A_{\widetilde A_n}\rtimes \mathbb Z$~\cite{Peifer}. 
%
%From this we obtain the following

We prove: 
\begin{theorem}\label{T:TildeA}
Let $n\geqslant 2$. Let $Z=A$ or $C$. 
\begin{itemize}
\item[(i)] The graph $\mathcal C_{parab}(\widetilde Z_n)$ is connected.
\item[(ii)] The graph $\mathcal C_{parab}(\widetilde Z_n)$ is isomorphic to a subgraph of $\mathcal {CG}(\mathbb D_{n+2})$.
\item[(iii)]  The graph $\mathcal C_{parab}(\widetilde Z_{n})$ has infinite diameter.
\item[(iv)] The graph $\mathcal C_{parab}(\widetilde Z_{n})$ is hyperbolic.
\end{itemize}
\end{theorem}

We already point out that part (i) can be proven using the same argument as given in  \cite[Lemma~5.2]{MorrisWright}. The embedding promised by part (ii) is described in Corollaries \ref{C:MapTheta} for $\widetilde A_n$ and~\ref{C:MapLambda} for~$\widetilde C_n$. 
Once~(ii) is proven, (iii) follows easily after showing that the isomorphic image of $\mathcal C_{parab}(\widetilde Z_n)$  is dense in $\mathcal{CG}(\mathbb D_{n+2})$ 
(Propositions \ref{P:DenseTildeA} and~\ref{P:DenseTildeC}). 
Finally, we observe that $\mathcal K_Z$ is \emph{not} quasi-isometrically embedded in $\mathcal{CG}(\mathbb D_{n+2})$
 --see Proposition~\ref{P:NotQI} and Remark \ref{R:NotQI}. 
 Therefore, hyperbolicity of $\mathcal C_{parab}(\widetilde Z_n)$ is not immediate and to establish it, we rely on a theorem of Kate Vokes \cite{vokes} which allows to prove the hyperbolicity of some subgraphs of the curve graph.
Whatsoever, it is important to note that our proofs of the hyperbolicity of $\mathcal C_{parab}(B_n)$, $\mathcal C_{parab}(\widetilde A_n)$ and $\mathcal C_{parab}(\widetilde C_n)$ strongly depend on the hyperbolicity of curve graphs; it would be highly desirable to obtain independent algebraic proofs.
% of the hyperbolicity of the graphs of irreducible parabolic subgroups. 

As explained in \cite{CalvezWiest}, Theorem \ref{T:B} can be rephrased by saying that the union $X_{NP}(B_n)$ of the normalizers of the proper irreducible \emph{standard} parabolic subgroups of $A_{B_n}$ is a hyperbolic structure on $A_{B_n}$, answering partially Conjecture 4.7 in \cite{CalvezWiest}. Finally, we exhibit another hyperbolic structure on~$A_{B_n}$ and partially answer Conjectures 4.2 and 4.18 in \cite{CalvezWiest}: 

\begin{theorem}\label{T:XP}
Let $n\geqslant 3$. Let $X_P(B_n)$ be the union of the proper irreducible standard parabolic subgroups of~$A_{B_n}$ and the center of $A_{B_n}$. Then $X_P(B_n)$ is a hyperbolic structure on $A_{B_n}$ which is not equivalent to $X_{NP}(B_n)$. 
\end{theorem}

The paper is arranged as follows. 
Section \ref{S:Prer} contains all the prerequisites for our results and some notation that will be used throughout the paper. In Section \ref{S:ArtinGroups}, we review some results on parabolic subgroups of Artin-Tits groups; in Section \ref{S:Correspondence}, we recall the  correspondence between essential simple closed curves in $\mathbb D_{n+1}$ and proper irreducible parabolic subgroups of $A_{A_n}$ and we introduce useful notation; in Section \ref{S:Vokes}, we present the above-mentionned theorem of Vokes \cite{vokes}. 
% which allows to establish the hyperbolicity of some subgraphs of the curve graph. 
Theorem \ref{T:B} is proved in Section~\ref{S:B}.
Sections \ref{S:TildeA} and \ref{S:TildeC} are devoted to the study of $\mathcal C_{parab}(\widetilde A_n)$ and $\mathcal C_{parab}(\widetilde C_n)$, respectively. 
%
%
%
%is devoted to the proofs of Theorems 
%\ref{T:TildeA} and~\ref{T:TildeAHyp} while Theorems \ref{T:TildeC} and \ref{T:TildeCHyp} are established in Section \ref{S:TildeC}.
Finally, in Section \ref{S:Comments} we prove Theorem \ref{T:XP}.
%
%open questions; in particular it is shown in Corollary \ref{Corollary} that the union of the normalizers of standard parabolic subgroups (the union of standard parabolic subgroups and the center, respectively) are hyperbolic structures on $A_{B_n}$ and that these structures are not equivalent --see \cite{CalvezWiest}. 

\section{Prerequisites}\label{S:Prer}

\subsection{Artin-Tits groups and Coxeter groups}\label{S:ArtinGroups}

Let $A_{\Gamma}$ be any Artin-Tits  group with standard generators $S$. Let $W_{\Gamma}=A_{\Gamma}/\langle\langle s^2\  |\ s\in S\rangle\rangle$ be the associated Coxeter group. 
The canonical projection $\pi: A_{\Gamma}\twoheadrightarrow W_{\Gamma}$ admits a \emph{set} section $\nu$ defined as follows --see for instance \cite[Theorem 3.3.1(ii)]{BjornerBrenti}. 
For $s\in S$, denote by $\bar s$ its image in~$W_{\Gamma}$ and $\bar S=\{\bar s\ |\ s\in S\}$. Let $w\in W_{\Gamma}$ and let $\bar s_1\ldots \bar s_r$ be a \emph{reduced expression} for $w$, meaning a shortest word representative for $w$ on $\bar S$; then $\nu(w)=s_1\ldots s_r$. 
%The \emph{pure} subgroup $PA_{\Gamma}$ of $A_{\Gamma}$ is the kernel of $\pi$: in type $A_n$ this is the pure braid group on $(n+1)$ strands $PA_{A_n}$.
%; we have the short exact sequence
%$$1\longrightarrow PA_{\Gamma} \hookrightarrow A_{\Gamma}\xrightarrow{\pi} W_{\Gamma}\longrightarrow 1.$$
The kernel of the projection $\pi$ is called the \emph{pure} Artin-Tits group (or coloured Artin-Tits group) and is denoted $PA_{\Gamma}$.

Given a subset $X$ of $S$, 
%we denote $\bar X=\{\bar x\  |\  x\in X\}$. 
the standard parabolic subgroup of $A_{\Gamma}$ generated by $X$ is denoted by $A_X$.
%; the subgroup of $W_{\Gamma}$ generated by $\bar X$ is denoted by $W_{\bar X}$.
 %it coincides with the Coxeter group $A_X/\langle\langle x^2, x\in X\rangle\rangle$ \cite{bourbaki}.

\begin{lemma}\cite[Theorem 4.1]{Paris}\label{L:ConjugacyParabolic}
Let $X,Y\subset S$.  The following are equivalent.
\begin{itemize}
\item[(i)] The subgroups $A_X$ and $A_Y$ are conjugate in $A_{\Gamma}$,
\item[(ii)] the sets $X,Y$ are conjugate in $A_{\Gamma}$.
%\item[(iii)] The sets $\bar X,\bar Y$ are conjugate in $W_{\Gamma}$,
%\item[(iv)] the subgroups $W_{\bar X}$ and $W_{\bar Y}$ are conjugate in $W_{\Gamma}$.
\end{itemize}
\end{lemma}

\begin{lemma}\cite[Corollary 4.2]{Paris}\label{L:Corollary}
Let $s,t\in S$; then $s$ and $t$ are conjugate in $A_{\Gamma}$ if and only if there is a path in $\Gamma$ which connects $s$ and $t$ and follows only edges with odd labels. 
\end{lemma}

In the rest of this section, we assume that $A_{\Gamma}$ is of \emph{spherical type}.
In this case, $W_{\Gamma}$ contains a unique longest element $w_0$ (see for instance \cite[Lemma 4.6.1]{Davis}). Denote its lift~$\nu(w_0)$ in~$A_{\Gamma}$ by~$\Delta_{\Gamma}$. Whenever $\Gamma$ is connected, it is known that the center of $A_{\Gamma}$ is cyclic generated by $\Delta_{\Gamma}$ or~$\Delta_{\Gamma}^2$ \cite[Theorem~7.2]{BS}. 
Any proper irreducible parabolic subgroup $P$ of $A_{\Gamma}$ is itself an irreducible Artin-Tits group of spherical type. The center of $P$ is a cyclic group generated by an element $z_P$ (actually we have the generators $z_P$ and $z_P^{-1}$ and we choose $z_P$ so that its exponent sum $\epsilon_{\Gamma}(z_P)$ is positive). We will always refer to this particular element as \emph{the central element} of $P$. 

The following two results will be used throughout the paper. The first one says in particular (with $g=1$) that the element $z_P$ determines completely the subgroup $P$ (and conversely). 
%The following version, due to Paris, will be sufficient for us as we only consider \emph{irreducible} parabolic subgroups; note however that the result has been later improved so as to include reducible parabolic subgroups \cite[Lemma 33]{Cumplido}. 

\begin{proposition}\cite[Theorem 5.2]{Paris}\label{P:Center}
Let $A_{\Gamma}$ be an Artin-Tits group of \emph{spherical type}. 
Let $P,Q$ be two irreducible parabolic subgroups of $A_{\Gamma}$ and let $g\in A_{\Gamma}$. 
Then $Q=P^g$ if and only if $z_Q=z_P^g$. 
\end{proposition}

The second result reduces the definition of adjacency in the graph of irreducible parabolic subgroups to a very simple commutation condition between the respective central elements. 

\begin{proposition}\cite[Theorem 2.2]{CGGMW}\label{P:Adjacent}
Let $A_{\Gamma}$ be an Artin-Tits group of \emph{spherical type}.
Let $P,Q$ be two distinct irreducible parabolic subgroups of $A_{\Gamma}$. Then $P,Q$ are adjacent (Definition \ref{D:CParab}) if and only if $z_P$ and $z_Q$ commute.
\end{proposition}

\subsection{Braids, curves and parabolic subgroups}\label{S:Correspondence}

%\subsection{Braids, curves and parabolic subgroups}\label{S:Curves}
%In this section we 
%will prove a key-lemma for both proofs of Propositions \ref{Prop:XNPHyp} and Proposition \ref{Prop:XPHyp}. 
%review a geometric perspective on braids and parabolic subgroups of $A_{A_n}$ and establish some useful lemmas.
%; although not strictly necessary, this should help providing a good picture for the proof.
Recall that the braid group on $(n+1)$ strands --or Artin-Tits group $A_{A_n}$-- can be identified with the Mapping Class Group of a closed disk with $(n+1)$ punctures $\mathbb D_{n+1}$. 
%, that is the group of isotopy classes of orientation-preserving automorphisms of $\mathbb D_n$ which induce the identity on the boundary of $\mathbb D_n$.
Assume that $\mathbb D_{n+1}$ is the closed disk in the complex plane of radius 
$\frac{n+2}{2}$ centered at $\frac{n+2}{2}$ and the punctures are at the integer numbers $1\leqslant i \leqslant n+1$. 
For $1\leqslant i\leqslant n$, the standard generator~$\sigma_i$ of $A_{A_n}$ corresponds to a clockwise half-Dehn twist along the horizontal segment $[i,i+1]$. 
The group $A_{A_n}$ naturally acts --on the right-- on the set of isotopy classes of essential simple closed curves in $\mathbb D_{n+1}$. 
In the sequel we will  
simply write
``essential curve'' or even ``curve'' instead of ``isotopy class of essential simple closed curve''; accordingly we will say that two distinct curves are disjoint if the corresponding isotopy classes admit disjoint representatives. 
%(curves without auto-intersection and enclosing at least 2 and at most $n$ punctures). 
The result of the action of a braid~$y$ on a curve $\mathcal C$ will be denoted by~$\mathcal C^y$.
Finally, note that a curve~$\mathcal C$ in~$\mathbb D_{n+1}$ divides the disk in two connected components naturally referred to as the \emph{interior} and the \emph{exterior} of $\mathcal C$.

Let $I$ be a \emph{proper subinterval} of $[n]=\{1,\ldots,n\}$, that is $$\emptyset\neq I\subsetneq [n],\ \  \left[(i<j<k) \wedge (i,k\in I)\right] \Longrightarrow j\in I.$$ 
This defines a proper irreducible standard parabolic subgroup of $A_{A_n}$, generated by $\{\sigma_i \ |\  i\in I\}$; denote this subgroup by~$A_I$. Let $m=\min(I)$ and $k=\#I$; the \emph{standard} or \emph{round curve associated to $I$} is the isotopy class of a geometric circle surrounding the $k+1$ punctures $m,\ldots, m+k$. 

As explained in \cite[Section 2]{CGGMW}, there is a one-to-one correspondence

$$\{ \text{curves in $\mathbb D_{n+1}$}\} \xrightarrow{\ \ \ \ \mathfrak f_n\ \ \ \ } \{\text{proper irreducible parabolic subgroups of $A_{A_n}$}\}$$

which induces a graph isomorphism $\mathcal {CG}(\mathbb D_{n+1})\longrightarrow \mathcal C_{parab}(A_{n})$.
%Given a curve in $\mathbb D_{n+1}$, the set of all isotopy classes of automorphisms of $\mathbb D_{n+1}$ whose support is enclosed by $\mathcal C$ is a parabolic subgroup.  
To a curve $\mathcal C$ in $\mathbb D_{n+1}$, we associate the subgroup $\mathfrak f_n(\mathcal C)$ of $A_{A_n}$ consisting of all isotopy classes of homeomorphisms of $\mathbb D_{n+1}$ whose support is enclosed by $\mathcal C$; this is a proper irreducible parabolic subgroup. 
In particular, given a proper subinterval $I$ of $[n]$, we have $\mathfrak f_n(\mathcal C_I)=A_I$. 
The inverse correspondence is given by the --well-defined-- formula $A_I^y\mapsto \mathcal C_I ^y$, for any proper subinterval $I$ of $[n]$ and any $y\in A_{A_n}$. 

Let us see that the adjacency condition given in Proposition \ref{P:Adjacent} turns $\mathfrak f_n$ into a  graph isomorphism. Let~$\mathcal C$ be a curve in~$\mathbb D_{n+1}$, let $P=\mathfrak f_n(\mathcal C)$ and let $z_P$ be  the central element of $P$. If $\mathcal C$ surrounds at least three punctures, then $z_P$ is the Dehn twist around the curve~$\mathcal C$. Otherwise, $z_P$ is the half-Dehn twist along an arc connecting the two punctures in the interior of $\mathcal C$ and which does not intersect $\mathcal C$. Now, given two parabolic subgroups $P_1=\mathfrak f_n(\mathcal C_1)$ and $P_2=\mathfrak f_n(\mathcal C_2)$, $z_{P_1}$ and $z_{P_2}$ commute if and only if $\mathcal C_1$ and $\mathcal C_2$ are disjoint. 

Before going on, we introduce a set of special braids which will play an important role in the paper. Let $n\geqslant 3$; let $p,q,r$ be positive integers with $1\leqslant p\leqslant q$ and $q+1\leqslant r\leqslant n+1$. 
We define 
$$\xi_{p,q,r}=\Pi_{i=q}^{r-1}\sigma_i\ldots \sigma_{i-(q-p)}.$$
In this positive braid, the strands numbered $p,\ldots, q$ end at positions $p+r-q,\ldots r$ without crossings between them and the strands numbered $q+1,\ldots, r$ end at positions $p,\ldots, p+r-q-1$ without crossings between them. As an example, Figure \ref{F:Xi}(i) shows $\xi_{3,4,8}\in A_{A_{9}}$.
%An example of such a braid $\xi_{p,q,r}$ is depicted in Figure \ref{F:Xi}. 

For $1\leqslant i \leqslant n$, define also $a_i=\sigma_i\ldots \sigma_1$, $b_i=\sigma_i\ldots \sigma_n$, and $a_0=b_{n+1}$ is the trivial braid. Equivalently, $a_i=\xi_{1,i-1,i}$ and $b_i=\xi_{i,i+1,n+1}$. 
For $y\in A_{A_n}$, we denote by $\pi_{y}$ the permutation in~$\mathfrak S_{n+1}$ associated to~$y$. Notice that $\pi_{a_i}(i+1)=1$, for all $0\leqslant i\leqslant n$ and $\pi_{b_i}(i)=n+1$, for all $1\leqslant i\leqslant n+1$.

\begin{lemma}\label{L:ActionOfai}
Let $I$ be a proper subinterval of $[n]$, $m_I=\min(I)$ and $k_I=\#I$, so that the circle~$\mathcal C_I$ in $\mathbb D_{n+1}$ surrounds the punctures $m_I$ to $m_I+k_I$. Let $0\leqslant i_0 \leqslant n$. 
\begin{itemize}
\item[(i)] If $i_0+1<m_I$, i.e. if the puncture $i_0+1$ is to the left of $\mathcal C_I$, then $\mathcal C_I^{a_{i_0}}=\mathcal C_I$.
\item[(ii)] If $i_0+1>m_I+k_I$, i.e. if the puncture $i_0+1$ is to the right of $\mathcal C_I$, then $\mathcal C_I^{a_{i_0}}=\mathcal C_{I'}$, where $I'=\{i+1\ |\  i\in I\}$.
\item[(iii)] If $m_I\leqslant i_0+1\leqslant m_I+k_I$, i.e. if the puncture $i_0+1$ is in the interior of $\mathcal C_I$, then $\mathcal C_I^{a_{i_0}}=\mathcal C_I^{a_{m_I-1}}$ is not standard (except if $m_I=1$) and $\mathcal C_I^{a_{i_0}\xi_{2,m_I,m_I+k_I}}=\mathcal C_{[1,k_I]}$. We will write $\xi_I= \xi_{2,m_I,m_I+k_I}$ (it does not depend on $i_0$).
\end{itemize}
\end{lemma}
\begin{proof}
The contents of Lemma \ref{L:ActionOfai} are depicted in Figure \ref{F:Xi}(ii)-(iii). Only the third case might need a short proof: it suffices to observe that the crossings $\sigma_{i_0},\ldots, \sigma_{m_I}$ fix the curve $\mathcal C_I$ as they are inner to it, so it only remains the action of $\sigma_{m_I-1}\cdots \sigma_1 = a_{m_I-1}$.
\end{proof}

\begin{figure}[hbt]
\center
\includegraphics[scale=0.8]{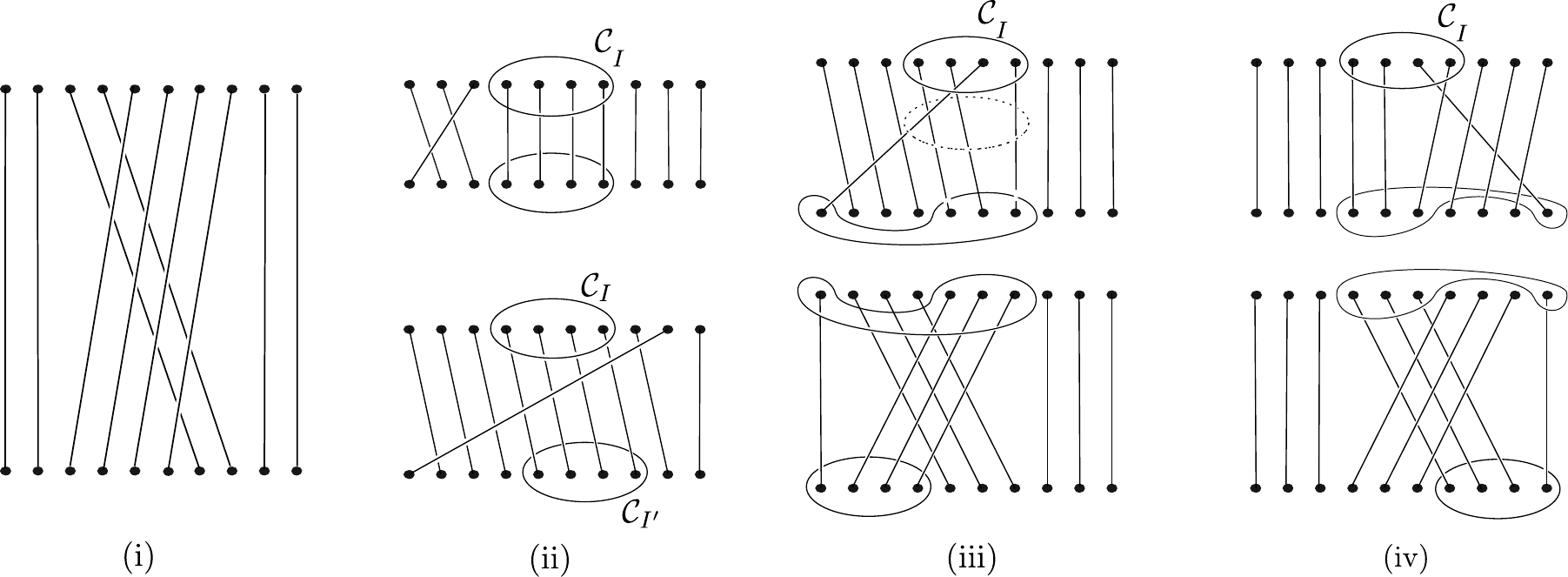}
\caption{
(i) The braid $\xi_{3,4,8}\in A_{A_9}$. (ii) Lemma \ref{L:ActionOfai}(i)-(ii): if the puncture $i_0+1$ is not surrounded by the curve $\mathcal C_I$, then $\mathcal C_I^{a_{i_0}}$ is standard again. (iii) Lemma \ref{L:ActionOfai}(iii): the puncture $i_0+1$ is surrounded by $\mathcal C_I$ (and $m_I>1$). The curve $\mathcal C_I$ is preserved by the action of the first crossings of~$a_{i_0}$; the action of the braid $\xi_I=\xi_{2, m_I, m_i+k_I}$ standardizes $\mathcal C_I^{a_{i_0}}$. 
(iv) Lemma \ref{L:ActionOfbi}(iii): the puncture $i_0$ is surrounded by $\mathcal C_I$ (and $m_I+k_I<n+1$). The action of the braid $\xi'_I=\xi_{m_I, m_i+k_I-1,n}$ standardizes the curve~$\mathcal C_I^{b_{i_0}}$.}
\label{F:Xi}
\end{figure}

\begin{lemma}\label{L:ActionOfbi}
Let $I$ be a proper subinterval of $[n]$, $m_I=\min(I)$ and $k_I=\#I$, so that the circle~$\mathcal C_I$ in $\mathbb D_{n+1}$ surrounds the punctures $m_I$ to $m_I+k_I$. Let $1\leqslant j_0 \leqslant n+1$. 
\begin{itemize}
\item[(i)] If $j_0<m_I$, i.e. if the puncture $j_0$ is to the left of $\mathcal C_I$, then $\mathcal C_I^{b_{j_0}}=\mathcal C_{I'}$, where $I'=\{i-1\ |\ i\in I\}$.
\item[(ii)] If $j_0>m_I+k_I$, i.e. if the puncture $j_0$ is to the right of $\mathcal C_I$, then $\mathcal C_I^{b_{j_0}}=\mathcal C_{I}$, where $I'=\{i+1\ |\  i\in I\}$.
\item[(iii)] If $m_I\leqslant j_0\leqslant m_I+k_I$, i.e. if the puncture $j_0$ is in the interior of $\mathcal C_I$, then $\mathcal C_I^{b_{j_0}}=\mathcal C_I^{b_{m_I+k_I}}$ is not standard (except if $m_I+k_I=n+1$) and $\mathcal C_I^{b_{j_0}\xi_{m_I,m_I+k_I-1,n}}=\mathcal C_{[n-k_I+1,n]}$. We will write $\xi'_I= \xi_{m_I,m_I+k_I-1,n}$ (it does not depend on $j_0$).

\end{itemize}
\end{lemma}
\begin{proof}
Similar to Lemma \ref{L:ActionOfai}. An example of (iii) is depicted in Figure \ref{F:Xi}(iv).
\end{proof}

\begin{remark}\label{R:Pure}\rm
We note that $\xi_I$ has its first strand straight and $\xi'_I$ has its last strand straight.
\end{remark}

 \subsection{Hyperbolicity for some graphs of curves}\label{S:Vokes}
In this section, we present a specialization of a theorem of Kate Vokes, which we will use as a criterion for proving the hyperbolicity of some subgraphs of the curve graph of the punctured disk. 
Consider the $n$-times punctured disk $\mathbb D_n$. 
A \emph{subsurface} of $\mathbb D_n$ is (the isotopy class of) a connected subsurface~$X$ of $\mathbb D_n$ so that every boundary component of $X$ is either $\partial \mathbb D_n$ or an essential curve in~$\mathbb D_{n}$. 
A simple closed curve in $X$ is \emph{essential} (in $X$) if it cannot be isotoped in $X$ to a point, a puncture or a boundary component of $X$. 
By an \emph{annulus} in $\mathbb D_n$ we mean the subsurface consisting of a tubular neighborhood of some essential curve in $\mathbb D_n$. Given a curve~$\mathcal C$ in $\mathbb D_n$ and a subsurface~$X$ of~$\mathbb D_n$ we say that $\mathcal C$ and $X$ are \emph{disjoint} if they admit disjoint representatives; $X$ is said to be a \emph{witness} for $\mathcal C$ if $\mathcal C$ and $X$ are not disjoint. In particular, $X$ is not a witness for any of its boundary components. Two subsurfaces are \emph{disjoint} if they admit disjoint representatives.

\begin{theorem}\cite[Corollary 1.5]{vokes}\label{T:Vokes}
Let $\hat{\mathcal C}$ be a family of curves in $\mathbb D_n$; let $\mathcal K$ be the subgraph of $\mathcal{CG}(\mathbb D_n)$ induced by $\hat{\mathcal C}$, equipped with the combinatorial metric $d_{\mathcal K}$ (each edge has length one). 
Let~$\mathcal X$ be the set of witnesses for $\mathcal K$, that is the set of all subsurfaces of $\mathbb D_n$ which are a witness for every element of $\hat{\mathcal C}$. 
Suppose: 
\begin{itemize}
\item[(i)] $\mathcal K$ is connected, 
\item[(ii)] The action of $PA_{A_{n-1}}$ on $\mathbb D_{n}$ induces an isometric action of $PA_{A_{n-1}}$ on $\mathcal K$.  
\item[(iii)] $\mathcal X$ contains no annulus.
\item [(iv)] No two elements of $\mathcal X$ are disjoint.  
\end{itemize}
Then $\mathcal K$ is hyperbolic. 
\end{theorem}

%\begin{remark}\rm
%Clauses (2) and (4) of  \cite[Definition 2.1]{vokes} are automatically satisfied under our hypothesis; namely, each vertex of $\mathcal K$ being a curve in $\mathbb D_{n}$ is in particular a multicurve and each pair of adjacent vertices in $\mathcal K$ are disjoint curves. In (ii), which is the adaptation of Clause (3) of \cite[Definition 2.1]{vokes}, we need to restrict attention to pure braids since mapping classes in \cite{vokes} are required to fix the boundary pointwise. 
%ctually, each puncture in~$\mathbb D_{n}$ should be treated as a boundary component and this is why we have to consider pure braids in~(ii). 
%However, as noted in \cite[Section 2.3]{MasurMinsky2}, both points of view are equivalent and it will be more convenient for us to maintain the difference between punctures of the disk and ``real'' boundaries, thinking of punctures as ``distinguished boundary components''. 
%\end{remark}

\begin{remark}\rm \label{R:Vokes}
Let us check that the hypothesis of Theorem \ref{T:Vokes} match the hypothesis of \cite[Corollary 1.5]{vokes}, namely that $\mathcal K$ is a \emph{twist-free multicurve graph} having no pair of disjoint witnesses. The second half is exactly our clause (iv). 
 The definition of a twist-free multicurve graph (\cite[Definition 2.1]{vokes}) consists of clauses~{(1)-(5).} 
To see clauses (2) and (4), observe that each vertex of~$\mathcal K$ being a curve in $\mathbb D_{n}$ is in particular a multicurve and that two adjacent vertices in $\mathcal K$ are disjoint curves. 
Clauses (1) and (5) correspond to~(i) and (iii) of Theorem \ref{T:Vokes}, respectively.
Clause (3) is adapted into clause (ii) of Theorem \ref{T:Vokes}.
A priori, the results in \cite{vokes} work for \emph{compact} surfaces (possibly with boundary). However, according to \cite[Section 2.3]{MasurMinsky2}, punctures can be treated as boundary components, so the results of \cite{vokes} apply to punctured surfaces as well. In the case of the punctured disc, we have to replace the whole braid group by the pure braid group since mapping classes in~\cite{vokes} are required to fix the boundary pointwise. 
In the sequel we will find it more convenient to maintain the difference between punctures of the disk and ``real'' boundaries, thinking of punctures as ``distinguished boundary components''. 
\end{remark}

\section{The graph $\mathcal C_{parab}(B_n)$}\label{S:B}

A proper irreducible standard parabolic subgroup of $A_{B_n}$ is determined by a proper 
subinterval of~$[n]$: for any proper subinterval $I$ of $[n]$, we denote by $B_I$ the proper irreducible standard parabolic subgroup of $A_{B_n}$ generated by $\{\tau_i\ |\  i\in I\}$. 
%{\color{red} Notice that in $A_{B_n}$, by Lemma \ref{L:Corollary}, the standard generators fall into two conjugacy classes, namely each $\tau_i,$ $i\geqslant 2$ is conjugate to each other and not conjugate to $\tau_1$. \marginpar{Useful??}}
%In view of Lemma \ref{L:ConjugacyParabolic}, each proper irreducible parabolic subgroup~$P$ of~$A_{B_n}$ is exactly one of the following types: 
%\begin{itemize}
%\item Type A if $P$ is conjugate to $B_I$ for $I\subset \{2,\ldots,n\}$, 
%\item Type B if $P$ is conjugate to $B_I$ with $1\in I$. 
%\end{itemize}

There is a monomorphism 
\begin{align*}
    \eta_n :   A_{B_n} & \longrightarrow A_{A_n} \\
          \tau_i & \longmapsto \begin{cases}  \sigma_1^2 & \text{if $i=1$,}\\
                                               \sigma_i  & \text{if $2\leqslant i\leqslant n$.}
\end{cases}
\end{align*}
The image of $\eta_n$ is the subgroup $\mathfrak P_1$ of $(n+1)$-strands \emph{1-pure braids}, that is the subgroup of all $(n+1)$-strands braids in which the first strand ends in the first position. In other words, a braid~$y$ on $(n+1)$ strands belongs to $\mathfrak P_1$ if and only if $\pi_y(1)=1$, where $\pi_y=\pi(y)$ is the permutation in $\mathfrak S_{n+1}=W_{A_n}$ associated to $y$. 
A presentation for $\mathfrak P_1$ was given by Wei-Liang Chow \cite{Chow} in 1948; for a proof that $\eta_n$ defines an isomorphism between $A_{B_n}$ and $\mathfrak P_1$, the reader may consult \cite{Peifer}.

For the rest of this section, given a proper subinterval $I$ of $[n]$, we shall denote $m_I=\min (I)$ and~$k_I=\# I$. 
The central elements $z_{A_I}$ of $A_I$ and $z_{B_I}$ of $B_I$ are given by the following formulae (see \cite[Lemmas 3.1 and 4.1]{Paris2}):

\begin{center}
\begin{tabular}{ccll}
$z_{A_I}$ &
 = &
$\left\{  
  \begin{tabular}{l} 
  
 $\sigma_{m_I}$ \\
$\left((\sigma_{m_I}\ldots \sigma_{m_I+k_I-1})\ldots (\sigma_{m_I}\sigma_{m_I+1})\sigma_{m_I}\right)^2$
\end{tabular}
 \right. $    
 
 & \begin{tabular}{l} 
  
 if $k_I=1$, \\
if $k_I\geqslant 2$,
\end{tabular}
 
 \\
 
  & & & \\

$z_{B_I}$ & = & 
$\left\{  
  \begin{tabular}{l} 
  
 $\tau_{m_I}$\\
 $\left((\tau_{m_I}\ldots \tau_{m_I+k_I-1})\ldots (\tau_{m_I}\tau_{m_I+1})\tau_{m_I}\right)^2$\\
 $(\tau_{k_I}\ldots \tau_2\tau_1\tau_2\ldots\tau_{k_I})\ldots (\tau_2\tau_1\tau_2) \tau_1 $
\end{tabular}
 \right. $

 & \begin{tabular}{l} 
  
if $k_I=1$, \\
if $k_I\geqslant 2$ and $1\notin I$,\\
if $k_I\geqslant 2$ and $1\in I$.
\end{tabular}
 
\end{tabular}
\end{center}

The proof of the next lemma follows from an easy computation left to the reader. 

\begin{lemma}\label{L:Comput}
Let $I$ be a proper subinterval of $[n]$. We have 
\begin{itemize}
\item[(i)] $\eta_n(B_I)=A_I\cap \mathfrak P_1$.
\item[(ii)] $\eta_n(z_{B_I})=z_{A_I}$, except if $I=\{1\}$, in which case $\eta_n(z_{B_I})=z_{A_I}^2$. 
\end{itemize}
\end{lemma}

%\begin{proof}
%Suppose first that $1\notin I$. If $I=\{m\}$ has only one element $m\geqslant 2$, then $\eta(c_{B,I})=\eta(\tau_m)=\sigma_m=c_{A,I}$. If $I$ has more than one element, $\eta(c_{B,I})=\eta(\Delta_{B,I}^2)=\Delta_{A,I}^2=c_{A,I}$. 
%
%If $1\in I$ and $I$ has more than one element, $\eta(c_{B,I})=\eta(\Delta_{B,I})=\sigma_1^2(\sigma_2\sigma_1^2\sigma_2)\ldots (\sigma_m\ldots \sigma_2\sigma_1^2\sigma_2\ldots \sigma_m)=\Delta_{A,I}^2=c_{A,I}^2$. Finally, 
%\end{proof}

\begin{proposition}\label{P:WellDef}
Let $I,J$ be proper subintervals of $[n]$ and let $g\in A_{B_n}$. 
The following are equivalent: 
\begin{itemize}
\item[(i)] $B_I^g=B_J $,
\item[(ii)]$\mathcal C_I^{\eta_n(g)}=\mathcal C_J$. 
\end{itemize}
\end{proposition}

\begin{proof} 
Note that using the isomorphism $\mathfrak f_n$ from Section \ref{S:Correspondence}, (ii) is equivalent to $A_I^{\eta_n(g)}=A_J$.

Assume (i). 
According to Proposition \ref{P:Center}, we have $z_{B_I}^g=z_{B_J}$. Assume first that $I=\{1\}$ (hence $J=\{1\}$, by Lemmas \ref{L:ConjugacyParabolic} and \ref{L:Corollary}) so that $\tau_1^g=\tau_1$, which yields $({\sigma_1^2})^{\eta_n(g)}={\sigma_1^2}$ after applying the monomorphism $\eta_n$. Then \cite[Theorem 2.2]{Zhu} ensures that also $\sigma_1^{\eta_n(g)}=\sigma_1$ whence $A_I^{\eta_n(g)}=A_J$. If on the contrary $I,J\neq\{1\}$, Lemma \ref{L:Comput}(ii) yields that $z_{A_I}^{\eta_n(g)}=z_{A_J}$ from which $A_I^{\eta_n(g)}=A_J$ follows using Proposition~\ref{P:Center}.

Conversely, assume (ii). By Proposition \ref{P:Center}, we get $z_{A_I}^{\eta_n(g)}=z_{A_J}$. By Lemma \ref{L:ConjugacyParabolic}, again $I=\{1\}$ if and only if $J=\{1\}$, as $\eta_n(g)$ is 1-pure. In this case it follows that $({\sigma_1^2})^{\eta_n(g)}={\sigma_1^2}$; as $\eta_n$ is injective, we get $\tau_1^g=\tau_1$, which is to say $B_I^g=B_J$. If $I\neq\{1\}$, from the relation~$z_{A_I}^{\eta_n(g)}=z_{A_J}$ and using Lemma \ref{L:Comput}(ii) we get $z_{B_I}^g=z_{B_J}$ which, by Proposition~\ref{P:Center}, implies~(i). 
\end{proof}

\begin{proposition}\label{P:AdjB}
Let $I,J$ be proper subintervals of $[n]$ and let $g\in A_{B_n}$. 
The following are equivalent: 
\begin{itemize}
\item[(i)] $B_I^g$ and $B_J $ are adjacent in $\mathcal C_{parab}(B_n)$,
\item[(ii)] $\mathcal C_I^{\eta_n(g)}$ and $\mathcal C_J$ are adjacent in $\mathcal{CG}(\mathbb D_{n+1})$.
\end{itemize}
\end{proposition}

\begin{proof} 
By Proposition \ref{P:Adjacent}, (i) is equivalent to saying that $z_{B_I}^g$ and $z_{B_J}$ commute; by 
Lemma~\ref{L:Comput}(ii) (and injectivity of $\eta_n$), this is equivalent to $z_{A_I}^{\eta_n(g)}$ and $z_{A_J}$ commuting (notice that~$\sigma_1^2$ and~$\sigma_1$ have the same centralizer in $A_{A_n}$ by \cite[Theorem 2.2]{Zhu}). Using Proposition \ref{P:Adjacent} again, this is equivalent in turn to $A_I^{\eta_n(g)}$ and $A_J$ being adjacent in $\mathcal C_{parab}(A_n)$. Using the isomorphism $\mathfrak f_n^{-1}$ from Section~\ref{S:Correspondence}, this is also equivalent to (ii).
\end{proof}

Our next goal is to show that each curve $\mathcal C$ in $\mathbb D_{n+1}$ can be written $\mathcal C= \mathcal C_I^{\eta_n(g)}$ for some proper subinterval $I$ of $[n]$ and some $g\in A_{B_{n}}$. 
 Observe that $\mathfrak P_1$ has index $(n+1)$ in $A_{A_n}$. 
 The braids~$a_i$ introduced in Section \ref{S:Correspondence} enumerate the cosets of $\mathfrak P_1$: 
%Two braids $y_1,y_2$ are in the same right-coset if and only if $\pi_{y_1}(1)=\pi_{y_2}(1)$. 
given $y\in A_{A_n}$, there is a unique $i\in\{0,\ldots,n\}$ so that $ya_i\in \mathfrak P_1$. 

\begin{proposition}\label{P:Surj}
Let $\mathcal C$ be a curve in $\mathbb D_{n+1}$. There exists $\alpha\in \mathfrak P_1$ such that 
$\mathcal C^{\alpha}$ is standard. 
\end{proposition}

\begin{proof}
Let $\zeta$ be any braid such that $\mathcal C^{\zeta}$ is a round curve, say $\mathcal C_J$, and write $m=\min(J)$ and $k=\# J$. 
Let $i_0=\pi_{\zeta}(1)-1$, in such a way that $\zeta a_{i_0}$ is 1-pure. Use Lemma \ref{L:ActionOfai}. If $i_0+1 <m$ or $i_0+1>m+k$, then $\mathcal C^{\zeta a_{i_0}}=\mathcal C_J^{a_{i_0}}$ is standard and we can take $\alpha=\zeta a_{i_0}$. 
%
%
%
%
% then $\mathcal C^{\zeta a_{i_0}}=\mathcal C_J^{a_{i_0}}=\mathcal C_J$ (Lemma~\ref{L:ActionOfai} (i)) and we can take $I=J$, $\alpha=\zeta a_{i_0}$. If $i_0+1>m+k$, then $\mathcal C^{\zeta a_{i_0}}=\mathcal C_J^{a_{i_0}}=\mathcal C_{J'}$, where $J'=\{j+1\ |\  j\in J\}$ (Lemma \ref{L:ActionOfai} (ii)) and we can take $I=J'$, $\alpha=\zeta a_{i_0}$. 
Otherwise, suppose that $m\leqslant i_0+1\leqslant m+k$. If $\mathcal C^{\zeta a_{i_0}}=\mathcal C_J^{a_{m-1}}$ is not standard, we have $m>1$ and $\mathcal C^{\zeta a_{i_0} \xi_J}=(\mathcal C_J^{a_{m-1}})^{\xi_J}=\mathcal C_{[1,k]}$, so we can take $\alpha=\zeta a_{i_0} \xi_J$, which is 1-pure as $\zeta a_{i_0}$ and~$\xi_J$ are 1-pure (Remark \ref{R:Pure}).
\end{proof}

%Recall from Section \ref{S:Correspondence} that we have a graph isomorphism $\mathcal C_{parab}(A_{n})\xrightarrow{\ \ \mathfrak f^{-1} \ \ } \mathcal{CG}(\mathbb D_{n+1})$ defined by the formula $A_I^y \mapsto \mathcal C_I^{y}$, for every proper subinterval $I$ of $[n]$ and every $y\in A_{A_{n}}$. 

Now, the following achieves the proof of Theorem \ref{T:B}:

\begin{corollary}\label{C:MapH}
The assignment $B_I^g\mapsto \mathcal C_I^{\eta_n(g)}$, where $I$ is a proper subinterval of $[n]$ and $g\in A_{B_n}$, defines a graph isomorphism $\mathfrak H_n$ from $\mathcal C_{parab}(B_n)$ to $\mathcal {CG}(\mathbb D_{n+1})$.
\end{corollary}
\begin{proof}
By Proposition \ref{P:WellDef}, $\mathfrak H_n$ is a well-defined injective map. By Proposition \ref{P:Surj}, $\mathfrak H_n$ is surjective. Moreover, by Proposition \ref{P:AdjB}, both $\mathfrak H_n$ and its inverse are graph
homomorphisms. 
\end{proof}

\section{The graph $\mathcal C_{parab}(\widetilde A_n)$}\label{S:TildeA} 
Let us start with a description of the proper irreducible standard parabolic subgroups of $A_{\widetilde A_n}$. We say that a proper subset $I$ of $\{0,\ldots,n\}$ is a \emph{proper cyclic subinterval} if 
$I$ is either a proper subinterval of $\{0,\ldots,n\}$ or the union of two proper subintervals of $\{0,\ldots,n\}$ of the form $[l,n]$ and~$[0,k]$, for some $k,l$ with $1\leqslant k+1<l\leqslant n$. A proper irreducible standard parabolic subgroup of $A_{\widetilde A_n}$ is determined by a  proper cyclic subinterval of $\{0,\ldots, n\}$:
for any proper cyclic subinterval~$I$ of $\{0,\ldots, n\}$, we denote by $\widetilde A_I$ the proper irreducible standard parabolic subgroup of $A_{\widetilde A_n}$ generated by 
$\{\widetilde \sigma_i\ |\  i\in I\}$.

According to \cite{Peifer}, there is a monomorphism

\begin{align*}
    \theta_n :   A_{\widetilde A_n} & \longrightarrow A_{B_{n+1}} \\
          \widetilde\sigma_i & \longmapsto \begin{cases} \tau_{i+1} & \text{if $i\geqslant1$,}\\
                                                  \tau_{n+1}^{-1}\ldots \tau_3^{-1}\tau_1\tau_2\tau_1^{-1}\tau_3\ldots \tau_{n+1}  & \text{if $i=0$.}
\end{cases}
\end{align*}

%\cite{Peifer} shows the following:
\begin{proposition}\cite{Peifer}\label{P:SemiDirect}
Let $\rho=(\tau_1\tau_2\ldots \tau_{n+1})^{-1}\in A_{B_{n+1}}$. 
\begin{itemize}
\item[(i)] $\rho^{-(n+1)}$ is the central element of $A_{B_{n+1}}$.
\item[(ii)] $\theta_n(\widetilde \sigma_i)^{\rho}=\theta_n(\widetilde \sigma_{i+1})$ for $0\leqslant i\leqslant n-1$ and 
$\theta_n(\widetilde \sigma_n)^{\rho}=\theta_n(\widetilde \sigma_0)$.
\item[(iii)] The group $A_{B_{n+1}}$ can be decomposed as the semi-direct product $A_{B_{n+1}}=\theta_n(\widetilde A_n)\rtimes\langle\rho\rangle$, where the action of $\rho$ is given by conjugation, as in (ii).
\end{itemize}
\end{proposition}

\begin{proposition}\label{P:Map}
If $P$ is a proper irreducible parabolic subgroup of $A_{\widetilde A_{n}}$, then $\theta_n(P)$ is a proper irreducible parabolic subgroup of $A_{B_{n+1}}$. 
\end{proposition}
\begin{proof}
Suppose first that $P$ is standard: $P=\widetilde A_I$ for some proper cyclic subinterval of $\{0,\ldots,n\}$. 
If~$I$ is a subinterval of $\{0,\ldots,n\}$ which does not contain $0$, then~$\theta_n(P)$ is the subgroup
 of~$A_{B_{n+1}}$ generated by $\{\tau_{i+1}\ |\ i\in I\}$, which is a proper irreducible standard parabolic subgroup. If~$I$ is a subinterval of $\{0,\ldots,n\}$ which contains~0, then in view of Proposition \ref{P:SemiDirect}(ii), ${\theta_n(\widetilde A_I)^{\rho}=\langle \tau_{i+2}, i\in I\rangle}$, whence $\rho$ conjugates~$\theta_n(\widetilde A_I)$ 
  to a proper irreducible standard parabolic subgroup of $A_{B_{n+1}}$. Similarly, if~$I$ is of the form $[l,n]\cup[0,k]$, with 
  $1\leqslant k+1<l\leqslant n$, then $(\theta_n(\widetilde A_I))^{\rho^{n-l+2}}=\langle \tau_{2},\ldots, \tau_{k+n-l+3}\rangle$,
   whence~$\theta_n(\widetilde A_I)$ is again conjugate to a proper irreducible standard parabolic subgroup of $A_{B_{n+1}}$. 
Finally, if~$P$ is not standard, there is some $g\in A_{\widetilde A_n}$ such that $P^g=P_0$ is standard; by the above discussion, $\theta_n(P_0)$ is a proper irreducible parabolic subgroup of $A_{B_{n+1}}$ and from $\theta_n(P)^{\theta_n(g)}=\theta_n(P_0)$, we deduce that~$\theta_n(P)$ is itself a proper irreducible parabolic subgroup of $A_{B_{n+1}}$. 
%Endow the circuit defining Coxeter graph for $A_{\widetilde A_n}$ (Figure \ref{Figure} (c)) with the counterclockwise orientation. Let $P$ be a standard irreducible proper parabolic subgroup of $A_{\widetilde A_n}$: $P$ is generated by a set $I$ of consecutive vertices along this circuit. Let $k$ be such that $\widetilde \sigma_k$ is the first vertex in $I$ and $l$ be such that $\widetilde \sigma_l$ is the last vertex in $I$.  Assume first $k=0$; then $P=\langle \widetilde \sigma_0,\ldots,\widetilde\sigma_l\rangle$ and $\theta(P)^{\rho}=\langle \theta(\widetilde \sigma_1),\ldots ,\theta(\widetilde \sigma_{l+1})\rangle=\langle \tau_2,\ldots,\tau_{l+2}\rangle$, showing that $\theta(P)$ is parabolic. Assume now that $k\geqslant 1$. If $l\geqslant k$, then the letters generating $P$ belong to $\{\widetilde \sigma_i, i\in \{1,\ldots,n\}\}$ and $\theta(P)$ is a standard parabolic of $A_{B_{n+1}}$. If on the contrary $l<k$ then $\theta(P)^{\rho^{n-k+2}}$ is a standard parabolic subgroup, whence $\theta(P)$ is parabolic. 
%
%Finally, if $P$ is not standard, there is $g\in A_{\widetilde A_n}$ so that $P^g$ is standard, say $P_0$; then $\theta(P)=\theta(P_0^{g^{-1}})=\theta(P_0)^{\theta(g)^{-1}}$, which is parabolic by the above discussion. 
\end{proof}

\begin{proposition}\label{P:AdjTildeA}
Let $P,Q$ be proper irreducible parabolic subgroups of $A_{\widetilde A_n}$; the following are equivalent:
\begin{itemize}
\item[(i)] $P$ and $Q$ are adjacent in $\mathcal C_{parab}(\widetilde A_n)$,
\item[(ii)] $\theta_n(P)$ and $\theta_n(Q)$ are adjacent in $\mathcal C_{parab}(B_{n+1})$.
 \end{itemize}
\end{proposition}

\begin{proof}
Recall that the adjacency is defined in Definition \ref{D:CParab}
Both implications follow easily in view of the injectivity of $\theta_n$.
\end{proof}

Following the scheme of Section \ref{S:B}, we now want to characterize those parabolic subgroups of~$A_{B_{n+1}}$ which are of the form $\theta_n(P)$ for some proper irreducible parabolic subgroup $P$ of $A_{\widetilde A_n}$. Recall the graph isomorphism $\mathfrak H_{n+1}$: $\mathcal C_{parab}(B_{n+1})\longrightarrow \mathcal{CG}(\mathbb D_{n+2})$ --see Corollary \ref{C:MapH}.  We will say that a proper irreducible parabolic subgroup $P$ of $A_{B_{n+1}}$ is a \emph{braid subgroup} if there is some $I\subset \{2,\ldots, n+1\}$ so that $P$ is conjugate to $B_I$. Notice that the non-cyclic braid subgroups are exactly the non-cyclic proper irreducible parabolic subgroups which are isomorphic to an Artin-Tits group $A_k, k\geqslant 2$, hence the name.

%Proposition \ref{P:Map} allows us to define a map $\Theta$ from the set of vertices of $\mathcal C_{parab}(\widetilde A_n)$ to the set of vertices of $\mathcal C_{parab}(B_{n+1})$, which sends a proper irreducible parabolic subgroup of $A_{\widetilde A_n}$ to its image under the monomorphism $\theta$. We first describe the image of $\Theta$. 

\begin{proposition}\label{P:Image}
Let $Q$ be a proper irreducible parabolic subgroup of $A_{B_{n+1}}$; the following are equivalent:

\begin{itemize}
\item[(i)] There exists a proper irreducible parabolic subgroup $P$ of $A_{\widetilde A_n}$ such that $Q=\theta_n(P)$,
\item[(ii)] $Q$ is a braid subgroup of $A_{B_{n+1}}$,
\item[(iii)] The curve $\mathcal C=\mathfrak H_{n+1}(Q)$ does not surround the first puncture of $\mathbb D_{n+2}$. 
\end{itemize}
\end{proposition}

\begin{proof}
(i) $\Longrightarrow$ (ii). Let $P$ be a proper irreducible parabolic subgroup of $A_{\widetilde A_n}$ with the claimed property.  Suppose first that $P$ is cyclic, conjugate to some $\langle \widetilde \sigma_i\rangle$; then by the formulae defining~$\theta_n$, $\theta_n(P)$ is conjugate to $\langle \sigma_2\rangle$ hence is a cyclic braid subgroup. Suppose then that~$P$ is not cyclic, that is $P$ is isomorphic to a braid group of type $A_k, k\geqslant 2$. Note that $\theta_n(P)$ is isomorphic to $P$; by Proposition \ref{P:Map}, it is an irreducible parabolic subgroup; hence it is a braid subgroup. 

(ii) $\Longleftrightarrow$ (iii). We have the following chain of equivalences. \\
$Q$ is a braid subgroup of $A_{B_{n+1}}$ $\Longleftrightarrow$ $Q=B_I^x$ for some $x\in A_{B_{n+1}}$ and some $I\subset \{2,\ldots, n+1\}$ $\Longleftrightarrow$ $\mathfrak H_{n+1}(Q)=\mathcal C_I^{\eta_{n+1}(x)}$ for some $x\in A_{B_{n+1}}$ and some $I\subset \{2,\ldots, n+1\}$ $\Longleftrightarrow$ $\mathcal C=\mathfrak H_{n+1}(Q)=\mathcal C_I^y$, for some~$y$ 1-pure and some $I\subset \{2,\ldots, n+1\}$ $\Longleftrightarrow$ $\mathcal C=\mathfrak H_{n+1}(Q)$ does not surround the first puncture. The right-to-left direction of the latter equivalence uses Proposition \ref{P:Surj} which allows to standardize any curve through a 1-pure braid.

(ii)$\Longrightarrow$ (i). We need to show that every braid subgroup $Q$ satisfies that $Q=\theta_n(P)$ for some proper irreducible parabolic subgroup $P$ of $A_{\widetilde A_n}$. 
Assume first that $Q$ is standard; that is $Q=B_I$, for $I\subset \{2,\ldots,n+1\}$.
Setting $I'=\{i-1\ |\  i\in I\}$, we see that $Q=\theta_n(\widetilde A_{I'})$, as desired. 
If~$Q$ is not standard, let $\zeta\in A_{B_{n+1}}$ be such that $Q^{\zeta}=B_I$ is standard, 
for $I\subset \{2,\ldots,n+1\}$ and let $I'=\{i-1\ |\ i\in I\}$ so that $Q^{\zeta}=B_I=\theta_n(\widetilde A_{I'})$. 
If $\zeta=\theta_n(x)$ for some $x\in A_{\widetilde A_n}$ 
we are done as $Q=\theta_n(\widetilde A_{I'})^{\theta_n(x)^{-1}}=\theta_n({\widetilde A_{I'}}^{x^{-1}})$ is the image of 
a proper irreducible parabolic subgroup of $A_{\widetilde A_n}$. 
If on the contrary $\zeta$ is not in the image of $\theta_n$, there is some $r\in \mathbb Z$ and $x\in A_{\widetilde A_n}$ such that $\zeta\rho^r=\theta_n(x)$ --see Proposition \ref{P:SemiDirect}(iii). 
Using Proposition \ref{P:SemiDirect}(ii), we have $Q^{\zeta\rho^r}=(Q^{\zeta})^{\rho^r}=(\theta_n(\widetilde A_{I'}))^{\rho^r}=\theta_n(\widetilde A_{\{i+r,\ |\ i\in I'\}})$, where in the last term all indices are taken modulo $n+1$. But this is equivalent to saying that $Q=\theta_n(\widetilde A_{\{i+r\ |\ i\in I'\}}^{x^{-1}})$, which achieves the proof.
\end{proof}

%
%
%Suppose that $Q=\theta(P)$, for some proper irreducible parabolic subgroup $P$ of $A_{\widetilde A_n}$. If~$P$ is cyclic, $P$ is conjugate in 
%$A_{\widetilde A_n}$ to $\langle \widetilde \sigma_i\rangle$ for some $0\leqslant i\leqslant n$; by definition of $\theta$,
%we then see that~$Q$ is conjugate to $\langle \tau_2\rangle$ (recall that in $A_{B_{n+1}}$ all $\langle \tau_{i}\rangle$, $i\geqslant 2$, are conjugate) so $Q$ is of type~A. 
%If $P$ is not cyclic, $P$ is isomorphic to a braid group $A_{A_k}$ for $2\leqslant k\leqslant n$. Note that $Q=\theta(P)$ and $P$ are isomorphic; therefore $Q$ must be of type A. 
% 
%(ii) $\Longrightarrow$ (iii). Suppose that $Q$ is of type A: there is 
%$g\in A_{B_{n+1}}$ such that $Q=B_I^g$, for $I\subset \{2,\ldots,n+1\}$. 
%Then we have 
%$$\mathcal C=\mathfrak f^{-1}\mathfrak H(Q)=\mathfrak f^{-1}(A_I^{\eta(g)})=\mathcal C_I^{\eta(g)}.$$ Notice that $\mathcal C_I$ does not surround the first puncture and that $\eta(g)$ is 1-pure. 
%It follows that $\mathcal C$ does not surround the first puncture either. 
%
%(iii)$\Longrightarrow$ (i). 

\begin{definition}\rm
Let  $\mathcal K_A$ be the subgraph of $\mathcal{CG}(\mathbb D_{n+2})$ induced by the curves which do not surround the first puncture of $\mathbb D_{n+2}$.  
\end{definition}

\begin{corollary}\label{C:MapTheta}
The assignment $P\mapsto \mathfrak H_{n+1}(\theta_n(P))$, where $P$ is a proper irreducible parabolic subgroup of $A_{\widetilde A_n}$, defines a graph isomomorphism $\Theta_n$ from $\mathcal C_{parab}(\widetilde A_n)$ to $\mathcal K_A$. In particular, we have $d_{\mathbb D_{n+2}}(\Theta_n(P),\Theta_n(P'))\leqslant d_{\widetilde A_n}(P,P')$ for all proper irreducible parabolic subgroups $P$ and $P'$ of $A_{\widetilde A_n}$. 
\end{corollary}
\begin{proof}
Recall that $\mathfrak H_{n+1}$ is the isomorphism from Corollary \ref{C:MapH}.
Because $\theta_n$ is injective and according to Proposition \ref{P:Map}, the formula $P\mapsto \mathfrak H_{n+1}(\theta_n(P))$ defines an injective map $\Theta_n$ from the set of vertices of $\mathcal C_{parab}(\widetilde A_n)$ to the set of curves in $\mathbb D_{n+2}$. By Proposition~\ref{P:Image}, the image of this map is the set of vertices of $\mathcal K_A$. By Proposition \ref{P:AdjTildeA}, both~$\Theta_n$ and its inverse are graph homomorphisms.
\end{proof}

\begin{proposition}\label{P:DenseTildeA} The subgraph $\mathcal K_A$ is 1-dense in $\mathcal{CG}(\mathbb D_{n+2})$; as a consequence, the graph $\mathcal C_{parab}(\widetilde A_n)$ has infinite diameter.
\end{proposition}

\begin{proof}
Let $\mathcal C$ be any curve in $\mathbb D_{n+2}$; we may suppose that $\mathcal C$ surrounds the first puncture. We need to find a curve $\mathcal C'$ in $\mathcal K_A$ 
such that $\mathcal C'$ and $\mathcal C$ are disjoint. Suppose first that $\mathcal C$ is 
standard; if it surrounds only the first two punctures --that is $\mathcal C=\mathcal C_{\{1\}}$-- we can take $\mathcal C'=\mathcal C_{\{3\}}$; otherwise $\mathcal C'=\mathcal C_{\{2\}}$ does the job. If $\mathcal C$ is not standard, by Proposition \ref{P:Surj}, there is a 1-pure braid $\alpha$ so that~$\mathcal C^{\alpha}$ is standard --and still surrounding the first puncture. By the above line of argument, there is $\mathcal C''$ in~$\mathcal K_A$ disjoint from $\mathcal C^{\alpha}$ and it suffices to take $\mathcal C'={\mathcal C''}^{\alpha^{-1}}$.

Let us show the second part of the statement. Let~$P$ be any vertex of $\mathcal C_{parab}(\widetilde A_n)$ and let $M>0$. We shall find a vertex $P'$ of $\mathcal C_{parab}(\widetilde A_n)$ so that ${d_{\widetilde A_n}(P,P')>M}$. We know that $\mathcal{CG}(\mathbb D_{n+2})$ has infinite diameter \cite[Proposition 4.6]{MasurMinsky1}; in particular there exists a curve $\mathcal C$ in~$\mathbb D_{n+2}$ so that ${d_{\mathbb D_{n+2}}(\Theta_n(P),\mathcal C)>M+1}$. As we have just seen, there is a curve $\mathcal C'$ from $\mathcal K_A$ such that ${d_{\mathbb D_{n+2}}(\mathcal C,\mathcal C')\leqslant 1}$. Let $P'$ be the proper irreducible parabolic subgroup of $A_{\widetilde A_n}$ such that $\Theta_n(P')=\mathcal C'$. We deduce by Corollary~\ref{C:MapTheta} that $d_{\widetilde A_n}(P,P')>M$.
\end{proof}

%
%We are now in position to prove Theorem \ref{T:TildeA}. 
%
%{\it{Proof of Theorem \ref{T:TildeA}}}. Part (i), the connectivity of $\mathcal C_{parab}(\widetilde A_n)$, can be proved with literally the same arguments as \cite[Lemma 5.2]{MorrisWright}.
%Part (ii) is exactly Proposition \ref{P:Isom}. Let us show (iii). 
%Let~$P$ be any vertex of $\mathcal C_{parab}(\widetilde A_n)$ and let $M>0$. We shall find a vertex $P'$ of $\mathcal C_{parab}(\widetilde A_n)$ so that ${d_{\widetilde A_n}(P,P')>M}$. We know that $C_{parab}(B_{n+1})$ has infinite diameter \cite[Corollary 4.13, Proposition 4.4]{CalvezWiest}; in particular there exists a vertex $Q$ of $\mathcal C_{parab}(B_{n+1})$ so that ${d_{B_{n+1}}(\theta(P),Q)>M+1}$. By Lemma~\ref{L:Dense}, there is a proper irreducible parabolic subgroup $P'$ of $A_{\widetilde A_n}$ such that ${d_{B_{n+1}}(Q,\theta(P'))\leqslant 1}$. It follows that $d_{B_{n+1}}(\theta(P),\theta(P'))>M$ and by Proposition \ref{P:Isom}, we deduce that $d_{\widetilde A_n}(P,P')>M$. \hfill $\Box$

It would be conceivable that the subgraph $\mathcal K_A$ is quasi-isometrically embedded in $\mathcal{CG}(\mathbb D_{n+2})$, which would imply the hyperbolicity of $\mathcal C_{parab}(\widetilde A_n)$. However, this is not the case, as we will now see. 
We first show that $\mathcal K_A$ matches hypothesis (i)-(iii) of Theorem \ref{T:Vokes}. 

\begin{lemma}\label{L:TwistFree}
(i) $\mathcal K_A$ is connected. (ii) The action of the pure braid group $PA_{A_{n+1}}$ on $\mathbb D_{n+2}$ induces an isometric action on $\mathcal K_A$. (iii) No annulus in $\mathbb D_{n+2}$ can be a witness for all vertices of~$\mathcal K_A$. 
\end{lemma}

\begin{proof}
(i) Recall that we already proved in the introduction that $\mathcal C_{parab}(\widetilde A_n)$ is connected.(ii) The restriction to the pure braids  of the natural action of $A_{A_{n+1}}$ on $\mathbb D_{n+2}$ provides a simplicial action of $PA_{A_{n+1}}$ on~$\mathcal K_A$. (iii) Given an essential curve $\mathcal C$ in $\mathbb D_{n+2}$, we will see that there always exists some curve $c$ in $\mathcal K_A$ which is disjoint from the annulus determined by $\mathcal C$. Assume first that $\mathcal C$ does not surround the first puncture, so that $\mathcal C$ is a curve in $\mathcal K_A$; then $\mathcal C$ itself can be isotoped so that it does not intersect the annulus it determines. 
Suppose on the contrary that $\mathcal C$ surrounds the first puncture; if the exterior of $\mathcal C$ contains at least 2 punctures, we can take $c$ to be any curve in the exterior of~$\mathcal C$. Otherwise the interior of $\mathcal C$ contains $n+1\geqslant 3$ punctures  and we can choose $c$ to be any curve surrounded by $\mathcal C$ and not enclosing the first puncture. 
\end{proof}

As the  next proposition is not needed in the sequel, its proof is only sketched and we refer the reader to \cite{vokes} for a precise statement of the results used throughout. We denote by $d_{\mathcal K_A}$ the distance in the graph $\mathcal K_A$. 

\begin{proposition}\label{P:NotQI}
The subgraph $\mathcal K_A$ is not quasi-isometrically embedded in $\mathcal{CG}(\mathbb D_{n+2})$. More precisely, given any $M>0$, there exists a pair of curves $a,b$ in $\mathbb D_{n+2}$ not surrounding the first puncture with the following properties: 
\begin{itemize}
\item $a,b$ are disjoint from the circle $\mathcal C_{\{1\}}$, so that $d_{\mathbb D_{n+2}}(a,b)\leqslant 2$,
\item  $d_{\mathcal K_A}(a,b)>M$.
\end{itemize}
\end{proposition}

\begin{proof}
In view of Lemma \ref{L:TwistFree} (see also Remark \ref{R:Vokes}) and according to \cite[Corollary 1.2]{vokes}, we have a \emph{distance formula} in $\mathcal K_A$ from which we can deduce the claim. 
Fix $M>0$. Fix any curve $a$ not surrounding the first puncture and disjoint from the circle~$\mathcal C_{\{1\}}$. Let $D$ be the subdisk in $\mathbb D_{n+2}$ enclosed by~$\mathcal C_{\{1\}}$ and let $X=\mathbb D_{n+2}\setminus D$. Note that the subsurface~$X$ is homeomorphic to a disk with $(n+1)$ punctures and 
is a witness for every vertex of~$\mathcal K_A$. 
Note also that $a$ is a curve in~$X$. Let $C_0$ be the constant associated to $\mathcal K_A$ by \cite[Corollary 1.2]{vokes}, let $C>C_0$ and let $K_1=K_1(C)$, $K_2=K_2(C)$ as given by \cite[Corollary 1.2]{vokes}. Define an element $f$ of $A_{A_{n+1}}$  by choosing a pseudo-Anosov mapping class of $X$ which fixes each puncture of $X$ (a pseudo-Anosov pure braid on $(n+1)$ strands) and doubling its first strand. Then~$f$ acts loxodromically on the curve graph of $X$ \cite[Proposition~4.6]{MasurMinsky1} and, choosing $b$ as the image of~$a$ under a sufficiently high power of $f$, 
we can arrange that the distance $d_X(a,b)$ (in the curve graph of $X$) is bigger than $\max\{C, MK_1+K_2\}$. Notice that~$b$ is disjoint from~$\mathcal C_{\{1\}}$ and that $b$ does not surround the first puncture in $\mathbb D_{n+2}$. The distance formula \cite[Corollary 1.2]{vokes} then says in particular that 
$K_1.d_{\mathcal K_A}(a,b)+K_2$ is bounded from below by a sum of positive terms to which $d_X(a,b)$ contributes. It follows in particular that 
$d_{\mathcal K_A}(a,b)\geqslant \frac{d_X(a,b)-K_2}{K_1}>M$, as desired. This finishes the proof of Proposition~\ref{P:NotQI}.
\end{proof}

We are now ready for proving the hyperbolicity of $\mathcal C_{parab}(\widetilde A_n)$. As $\mathcal C_{parab}(\widetilde A_n)$ is isomorphic to~$\mathcal K_A$, it suffices to prove that $\mathcal K_A$ is hyperbolic. This will follow from Theorem \ref{T:Vokes} after we check the remaining hypothesis (iv). The next lemma describes all possible witnesses for $\mathcal K_A$; its proof will achieve the demonstration. Throughout, $p_1$ denotes the first puncture of $\mathbb D_{n+2}$. 

\begin{lemma}\label{L:Witness}
Let $X$ be a subsurface of $\mathbb D_{n+2}$. Then $X$ is a witness for $\mathcal K_A$ if and only if one of the following holds.

\begin{itemize}
\item[(i)] $X=\mathbb D_{n+2}$ or $X=\mathbb D_{n+2}\setminus D$, where $D$ is the interior of an essential curve surrounding $p_1$ and exactly one other puncture.  
\item[(ii)] $X$ is the interior of an essential curve surrounding $p_1$ and $n$ other punctures.
\item[(iii)] $X=X'\setminus D$, where $X'$ is the interior of an essential curve surrounding $p_1$ and $n$ other punctures  and $D$ is the interior of an essential curve surrounding $p_1$ and exactly one other puncture.  
\end{itemize}
We will say that $X$ is a witness of type (i), (ii) or (iii). Two witnesses for $\mathcal K_A$ are never disjoint.
\end{lemma}

\begin{proof}
The three types of subsurfaces in Lemma \ref{L:Witness} are depicted in Figure \ref{F:Witness}. 
First, we check that all subsurfaces (i)-(iii) are witnesses for $\mathcal K_A$: we see that the only curves which can fail to be witnessed by $X$ must surround the first puncture.
%
%Suppose that $X$ is of type (i). We may assume that $X\neq \mathbb D_{n+2}$, so that $X=\mathbb D_{n+2}\setminus D$. The only curve in $\mathbb D_{n+2}$ which can be realized disjointly from~$X$ is the boundary of $D$, which is not a curve in $\mathcal K'$. Suppose that $X$ is of type (ii): again, the only curve in $\mathbb D_{n+2}$ which can be realized disjointly from $X$ is  the boundary of $X$, which is not a curve in $\mathcal K'$. Suppose that $X=X'\setminus D$ is of type (iii): only two curves in $\mathbb D_{n+2}$ can be realized disjointly from~$X$, namely the respective boundaries of $D$ and $X'$; however none is a curve in $\mathcal K'$.  
\begin{figure}
\center
\includegraphics[scale=1]{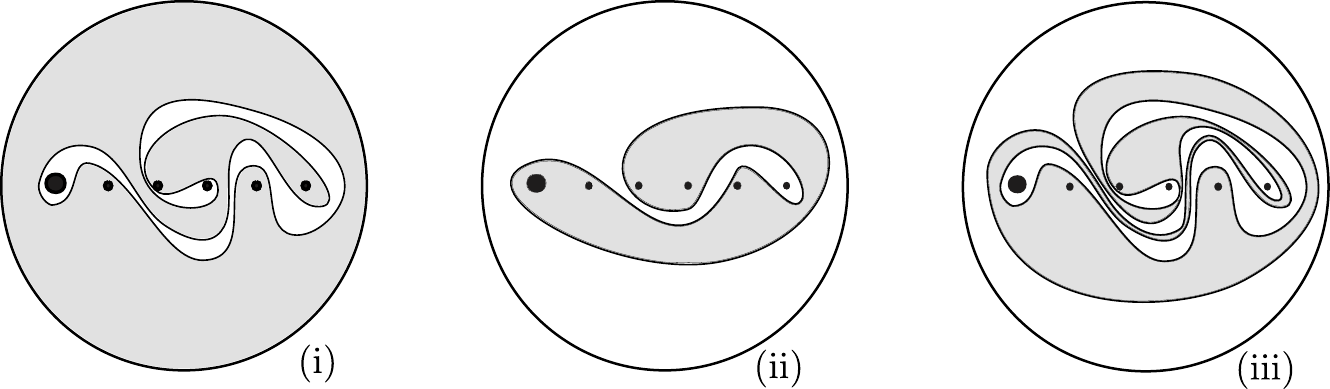}
\caption{Shaded, the different types of witnesses for $\mathcal K'$, the big dot represents $p_1$.}\label{F:Witness}
\end{figure}
Conversely, let $X$ be a witness for $\mathcal K_A$. We shall distinguish two cases. 

{\it{First case.}} Suppose that $\partial \mathbb D_{n+2}$ is a boundary component of~$X$. Assume that $X\neq \mathbb D_{n+2}$. Therefore, there is at least some essential curve $\mathcal C$ of $\mathbb D_{n+2}$ which is a boundary component of~$X$. 
Assume that $\mathcal C'$ is another essential curve of $\mathbb D_{n+2}$ which is a boundary component of $X$. Notice that~$\mathcal C$ and~$\mathcal C'$ cannot be nested as $X$ has to be connected. Then at least one of $\mathcal C$ or $\mathcal C'$ does not surround $p_1$ and this provides a particular curve of $\mathcal K_A$ for which $X$ is not a witness, a contradiction. Therefore~$X$ has exactly one essential curve $\mathcal C$ of $\mathbb D_{n+2}$ as a boundary component and~$\mathcal C$ must surround $p_1$. Moreover, $\mathcal C$ must surround exactly 2 punctures, otherwise there 
would exist a curve~$c$ in the interior of $\mathcal C$, not surrounding $p_1$, and $X$ would fail to be a witness for this curve~$c$. Letting~$D$ be the interior of $\mathcal C$, we have shown that whenever~$X$ has $\partial \mathbb D_{n+2}$ as a boundary component, $X=\mathbb D_{n+2}\setminus D$ has to be of type (i). 

{\it{Second case.}} Suppose that the boundary $\partial \mathbb D_{n+2}$ is not a boundary component of $X$. As $X$ is connected, $X$ has exactly one outermost boundary component which is an essential curve $\mathcal C$ of~$\mathbb D_{n+2}$. Again, $\mathcal C$ must surround $p_1$, otherwise $\mathcal C$ is a curve from $\mathcal K_A$ which is disjoint from $X$ and~$X$ fails to be a witness for $\mathcal K_A$. Moreover, $\mathcal C$ must surround $(n+1)$ punctures, otherwise the exterior of $\mathcal C$ would contain at least two punctures and there would exist a curve~$c$ from $\mathcal K_A$ entirely contained in $\mathbb D_{n+2}\setminus X$, contradicting that $X$ is a witness for $\mathcal K_A$. If $X$ has no other boundary component, we have shown that $X$ is of type (ii). 

Finally, suppose that $X$ has another boundary component. This must be an essential curve $\mathcal C'$ of~$\mathbb D_{n+2}$ which is nested in $\mathcal C$. Let $\mathcal C''$ be another putative boundary component of $X$ nested in $\mathcal C$. Then $\mathcal C'$ and $\mathcal C''$ cannot be nested, as $X$ is connected. Therefore only one of $\mathcal C'$, $\mathcal C''$ can surround~$p_1$: one of $\mathcal C',\mathcal C''$ is a curve from $\mathcal K_A$ for which $X$ is not a witness, contradiction. Therefore there is exactly one boundary component $\mathcal C'$ of $X$ nested in $\mathcal C$ and $\mathcal C'$ must surround $p_1$. Moreover, $\mathcal C'$ must surround exactly 2 punctures, for the same reasons as in the first case. Taking $X'$ to be the interior of $\mathcal C$ and $D$ to be the interior of $\mathcal C'$, we have shown that $X=X'\setminus D$ is of type (iii).   

Finally, the last claim follows from a direct case-by-case inspection.
%, using the following observations. 
%
%\begin{itemize}
%\item[(a)] Let $X$ and $Y$ be subsurfaces, let $\mathcal C$ be a boundary component of $X$ and $\mathcal C'$ be a boundary component of $Y$; if $\mathcal C$ and $\mathcal C'$ are not disjoint, then $X$ and $Y$ are not disjoint. 
%%\item[(b)] Let $X$ be a subsurface and let $\mathcal C$ be a boundary component of $X$. If $\mathcal C'$ is a curve which cannot be realized disjointly from $\mathcal C$, then $\mathcal C'$ cannot be realized disjointly from $X$. 
%\item[(b)] If $\mathcal C$ and $\mathcal C'$ are two curves whose interior contains $p_1$ and exactly one other puncture,  then~$\mathcal C$ and $\mathcal C'$ are not disjoint unless they are equal. 
%\item[(c)] If $\mathcal C$ and $\mathcal C'$ both surround $n+1$ punctures, then $\mathcal C$ and $\mathcal C'$ are not disjoint unless they are equal.
%\end{itemize} 
\end{proof}

The proof of the statements of Theorem \ref{T:TildeA} concerning $\mathcal C_{parab}(\widetilde A_n)$ is now complete: (i) was seen in the introduction, (ii) is the statement of Corollary \ref{C:MapTheta}, (iii) is the statement of Proposition \ref{P:DenseTildeA} and (iv) results from Lemmas \ref{L:TwistFree}, \ref{L:Witness} and Theorem \ref{T:Vokes}. 

We conclude with a generalization of Propositions \ref{P:Center} and \ref{P:Adjacent} to the case of the Artin-Tits group~$A_{\widetilde A_n}$. To the best of our knowledge, this was not previously written in the literature. Before proceeding, notice that if $P$ is a proper irreducible parabolic subgroup of $A_{\widetilde A_n}$, $P$ is an irreducible Artin-Tits group of \emph{spherical} type; therefore we can define the central element $z_P$ of~$P$, as in Section \ref{S:ArtinGroups}.  

\begin{proposition}\label{P:CenterATilde}
Let $P$ and $Q$ be two proper irreducible parabolic subgroups of $A_{\widetilde A_n}$. Let~${g\in A_{\widetilde A_n}}$. 
Then $P^g=Q$ if and only if $z_P^g=z_Q$.
\end{proposition}

\begin{proof}
The direct implication is obvious, considering the centers. Let us prove the converse. Given any proper irreducible parabolic subgroup $P$ of $A_{\widetilde A_n}$, Proposition \ref{P:Map} says that $\theta_n(P)$ is a proper irreducible subgroup of $A_{B_{n+1}}$. As $P$ and $\theta_n(P)$ are isomorphic, $\theta_n(z_P)$ generates the center of~$\theta_n(P)$. By the formulae defining $\theta_n$, we see that the exponent sums $\epsilon_{B_{n+1}}(\theta_n(x))$ and~$\epsilon_{\widetilde A_n}(x)$ coincide, for every $x\in A_{\widetilde A_n}$. We deduce that $\theta_n(z_P)=z_{\theta_n(P)}$, for every proper irreducible subgroup of $A_{\widetilde A_n}$. Now, assume $z_P^g=z_Q$. We have on the one hand $\theta_n(z_Q)=z_{\theta_n(Q)}$ while on the other hand, 
$$\theta_n(z_P^g)=\theta_n(z_P)^{\theta_n(g)}=z_{\theta_n(P)}^{\theta_n(g)}=z_{\theta_n(P)^{\theta_n(g)}}=z_{\theta_n(P^g)}.$$
We deduce $z_{\theta_n(P^g)}=z_{\theta_n(Q)}$ whence by Proposition \ref{P:Center}, $\theta_n(P^g)=\theta_n(Q)$ and by injectivity of $\theta_n$, $P^g=Q$ as desired. 
\end{proof}

\begin{proposition}\label{P:AdjATilde2}
Let $P$ and $Q$ be two distinct proper irreducible parabolic subgroups of $A_{\widetilde A_n}$. Then 
$P$ and $Q$ are adjacent if and only if $z_P$ and $z_Q$ commute.
\end{proposition}

\begin{proof}
Suppose that $P$ and $Q$ are adjacent --see Definition \ref{D:CParab}. Suppose $P\subset Q$; then as $z_Q$ is central in $Q$, $z_P$ and $z_Q$ commute. Similarly if $Q\subset P$. 
Otherwise, 
$P\cap Q=\{1\}$ and $pq=qp$ for each $p\in P$ and $v\in Q$; in particular, $z_{P}$ and $z_{Q}$ commute. 

Conversely, suppose that $z_P$ and $z_Q$ commute. Then $\theta_n(z_P)=z_{\theta_n(P)}$ and $\theta_n(z_Q)=z_{\theta_n(Q)}$ commute, whence by Proposition \ref{P:Adjacent}, $\theta_n(P)$ and $\theta_n(Q)$ are adjacent in $\mathcal C_{parab}(B_{n+1})$. By Proposition~\ref{P:AdjTildeA}, this implies that $P$ and $Q$ are adjacent. 
\end{proof}

\section{The graph $\mathcal C_{parab}(\widetilde C_n)$}\label{S:TildeC}
Let us start with a description of the proper irreducible standard parabolic subgroups of $A_{\widetilde C_n}$. A proper irreducible standard parabolic subgroup of $A_{\widetilde C_n}$ is determined by a proper subinterval of $[n+1]$: for any proper subinterval $I$ of~$[n+1]$, we denote by $\widetilde C_I$ the proper irreducible standard parabolic subgroup of $A_{\widetilde C_n}$ generated by ${\{\widetilde \tau_i\ |\  i\in I\}}$. 

%{\color{red} Notice that, by Lemma \ref{L:Corollary}, the standard generators of $A_{\widetilde C_n}$ fall in three conjugacy classes, namely $\widetilde \tau_1$ ($\widetilde \tau_{n+1}$, respectively) is the unique standard generator in its conjugacy class while each $\widetilde \tau_i$, $2\leqslant i\leqslant n$, is conjugate to each other. }
%{\color{blue} In view of Lemma \ref{L:ConjugacyParabolic}, each parabolic subgroup $P$ of $A_{\widetilde C_n}$ is of exactly one of the following types: 
%\begin{itemize}
%\item Type A if $P$ is conjugated to $\widetilde C_I$ for $I\subset \{1,\ldots,n-1\}$, 
%\item Type B1 if $P$ is conjugated to $\widetilde C_I$ with $0\in I$,
%\item Type B2 if $P$ is conjugate to $\widetilde C_I$ with $n\in I$.
%\end{itemize}}

The following facts can be found in \cite[Section 4]{Allcock}. There is a monomorphism 
\begin{align*}
    \lambda_n :   A_{\widetilde C_n} & \longrightarrow A_{A_{n+1}} \\
          \widetilde\tau_i & \longmapsto \begin{cases} \sigma_i^2 & \text{if $i=1$ or $i=n+1$},\\
          \sigma_{i} & \text{if $2\leqslant i\leqslant n$,}
\end{cases}
\end{align*}
The image of $\lambda_n$ is the subgroup $\mathfrak P$ of $(n+2)$-strands braids in which the first strand ends in the first position and the $(n+2)$nd strand ends in the $(n+2)$nd position. In other words, a braid $y$ on $(n+2)$ strands 
belongs to $\mathfrak P$ if and only if $\pi_y(1)=1$ and $\pi_y(n+2)=n+2$, where
$\pi_y=\pi(y)$ is the permutation in $\mathfrak S_{n+2}=W_{A_{n+1}}$ associated to $y$. We shall call these braids $(1,(n+2))$-\emph{pure}.

Although $A_{\widetilde C_n}$ is not of spherical type, we observe that each proper irreducible parabolic subgroup of $A_{\widetilde C_n}$ is an irreducible Artin-Tits group of spherical type. This allows to associate to each proper irreducible parabolic subgroup $P$ of $A_{\widetilde C_n}$ its central element $z_P$, as in Section \ref{S:ArtinGroups}.
Given a proper subinterval $I$ of $[n+1]$, the formula for $z_{\widetilde C_I}$ is very similar to the formulae 
given in Section \ref{S:B} and we do not write it explicitly. The following is the analogue of Lemma \ref{L:Comput}:

\begin{lemma}\label{L:Comput2}
Let $I$ be a proper subinterval of $[n+1]$. Then 
\begin{itemize}
\item[(i)] $\lambda_n(\widetilde C_I)=A_{I}\cap \mathfrak P$, 
\item[(ii)] $\lambda_n(z_{\widetilde C_I})=\begin{cases} z_{A_{I}} & \text{if $I\neq \{1\}, \{n+1\}$}\\
z_{A_{I}}^2 & \text{if $I=\{1\}$ or $I=\{n+1\}$}.
\end{cases}$
\end{itemize}
\end{lemma}

As $A_{\widetilde C_n}$ is not of spherical type, we do not know a priori the analogues of Propositions \ref{P:Center} and~\ref{P:Adjacent}. However, these analogues hold, as shown in the next Propositions \ref{P:WellDefTildeC} and~\ref{P:AdjTildeC} --compare Propositions \ref{P:CenterATilde} and \ref{P:AdjATilde2}.

\begin{proposition}\label{P:WellDefTildeC}
Let $I,J$ be proper subintervals of $[n+1]$ and let $g\in A_{\widetilde C_n}$. The following three statements are equivalent. 
\begin{itemize}
\item[(i)] $\widetilde C_{I}^g=\widetilde C_J$,
\item[(ii)] $z_{\widetilde C_I}^g=z_{\widetilde C_J}$,
\item[(iii)] $\mathcal C_{I}^{\lambda_n(g)}=\mathcal C_{J}$.
\end{itemize}
\end{proposition}

\begin{proof}
Note that using the isomorphism $\mathfrak f_{n+1}$ from Section \ref{S:Correspondence}, (iii) is equivalent to $A_I^{\lambda_n(g)}=A_J$. 
Assume (i). Considering the center, we get immediately (ii).

Let us show that (ii) implies (iii). Assume first that $I=\{1\}$. 
Then $z_{\widetilde C_J}=z_{\widetilde C_I}^g=\widetilde \tau_1^g$. By Lemma~\ref{L:Corollary}, this forces $J=\{1\}$, so that $\widetilde \tau_1^g=\widetilde \tau_1$.
%
%
%
%Recall that $\epsilon=\epsilon_{\widetilde C_n}$ is the exponent sum of words on $\{\widetilde \tau_1,\ldots,\widetilde \tau_{n+1}\}$. Note that $\epsilon(z_{\widetilde C_I}^g)=1$. Now, observe that for any subinterval $K$ of $[n+1]$, $\epsilon(z_{\widetilde C_K})=1$ if and only if $\#K=1$. From $z_{\widetilde C_I}^g=z_{\widetilde C_J}$ we deduce $\epsilon(z_{\widetilde C_J})=1$ and $\#J=1$. Therefore $z_{\widetilde C_J}=\widetilde \tau_i$ for some $i\in [n+1]$. We then have $\widetilde \tau_1^g=\widetilde \tau_i$. By Lemma \ref{L:Corollary}, $i=1$ and we deduce $\widetilde \tau_1^g=\widetilde \tau_1$. 
It follows that $(\sigma_1^2)^{\lambda_n(g)}=\sigma_1^2$ and from \cite[Theorem~2.2]{Zhu} we deduce that $\sigma_1^{\lambda_n(g)}=\sigma_1$. Finally, $A_{\{1\}}^{\lambda_n(g)}=A_{\{1\}}$ as desired. 
The proof is similar if we assume $I=\{n+1\}$. 
Suppose then that $I,J\neq \{1\},\{n+1\}$. By Lemma \ref{L:Comput2}(ii) we have $z_{A_{I}}^{\lambda_n(g)}=z_{A_{J}}$ and by Proposition \ref{P:Center}, we obtain $A_{I}^{\lambda_n(g)}=A_{J}$ as desired. 

Finally, assume (iii) and let us show (i). We have, as $\lambda_n(g)\in \mathfrak P$ and using Lemma \ref{L:Comput2}(i), 
$$\lambda_n(\widetilde C_I^g)=(\lambda_n(\widetilde C_I))^{\lambda_n(g)}=(A_{I}\cap \mathfrak P)^{\lambda_n(g)}=A_{I}^{\lambda_n(g)}\cap \mathfrak P^{\lambda_n(g)}=A_{J}\cap \mathfrak P=\lambda_n(\widetilde C_J)$$
and the injectivity of $\lambda_n$ ensures that $\widetilde C_I^g=\widetilde C_J$, as required. 
\end{proof}

%We can then define an injective map $\Lambda_n$ from the set of proper irreducible parabolic subgroups of $A_{\widetilde C_n}$ to the set of proper irreducible parabolic subgroups of $A_{A_{n+1}}$ given by the well-defined formula $\Lambda_n(\widetilde C_I^g)=A_{I'}^{\lambda_n(g)}$, for any proper subinterval $I$ of $\{0,\ldots,n\}$, $g\in A_{\widetilde C_n}$ and $I'=\{i+1,i\in I\}$. 

\begin{proposition}\label{P:AdjTildeC}
Let $I,J$ be proper subintervals of $[n+1]$ and let $g\in A_{\widetilde C_n}$. The following are equivalent. 
\begin{itemize}
\item[(i)] $\widetilde C_I^g$ and $\widetilde C_J$ are adjacent in $\mathcal C_{parab}(\widetilde C_n)$.
\item[(ii)] $z_{\widetilde C_I}^g$ and $z_{\widetilde C_J}$ commute.
\item[(iii)] $\mathcal C_{I}^{\lambda_n(g)}$ and $\mathcal C_{J}$ are adjacent in $\mathcal {CG}(\mathbb D_{n+2})$.
\end{itemize}
\end{proposition}

\begin{proof}
(i) $\Longrightarrow$ (ii) is proven in the same way as the direct implication of Proposition \ref{P:AdjATilde2}. 

Suppose (ii). By Lemma \ref{L:Comput2}(ii) (and \cite[Theorem 2.2]{Zhu} for the case when $I$ or $J=\{1\}$ or $\{n+1\}$), we obtain that $z_{A_{I}}^{\lambda_n(g)}$ and $z_{A_{J}}$ commute, which is to say, according to Proposition \ref{P:Adjacent}, that $A_{I}^{\lambda_n(g)}$ and $A_{J}$ are adjacent in $\mathcal C_{parab}(A_{n+1})$. Applying the isomorphism $\mathfrak f_{n+1}^{-1}$, we obtain (iii).

Suppose (iii) and let us show (i). Using the isomorphism $\mathfrak f_{n+1}$, $A_I^{\lambda_n(g)}$ and $A_J$ are adjacent. Suppose first that $A_{I}^{\lambda_n(g)}\subset A_{J}$. Then we have, as $\lambda_n(g)\in \mathfrak P$ and using Lemma \ref{L:Comput2}(i), 
$$\lambda_n(\widetilde C_I^g)=\lambda_n(\widetilde C_I)^{\lambda_n(g)}=(A_{I}\cap \mathfrak P)^{\lambda_n(g)}=A_{I}^{\lambda_n(g)}\cap \mathfrak P\subset A_{J}\cap \mathfrak P=\lambda_n(\widetilde C_J).$$
It follows that $\lambda_n(\widetilde C_{I}^g)\subset \lambda_n(\widetilde C_{J})$ and injectivity of $\lambda_n$ shows that $\widetilde C_I^g\subset \widetilde C_J$. The proof when ${A_{J}\subset A_{I}^{\lambda_n(g)}}$ is similar. 
Assume finally that $A_{I}^{\lambda_n(g)}\cap A_{J}=\{1\}$ and that both subgroups commute. Then 
we have, as $\lambda_n(g)\in \mathfrak P$ and using Lemma \ref{L:Comput2}(i) and the injectivity of $\lambda_n$,
$$\lambda_n(\widetilde C_I^g\cap \widetilde C_J)=\lambda_n(\widetilde C_I^g)\cap \lambda_n(\widetilde C_J)=(A_{I}^{\lambda_n(g)}\cap \mathfrak P)\cap (A_{J}\cap \mathfrak P)=A_{I}^{\lambda_n(g)}\cap A_{J}\cap \mathfrak P=\{1\}$$ and $\widetilde C_I^g\cap \widetilde C_J=\{1\}$. As $\lambda_n(\widetilde C_{I}^g)$ is a subgroup of $A_{I}^{\lambda_n(g)}$, $\lambda_n(\widetilde C_J)$ is a subgroup of $A_{J}$ and $A_{I}^{\lambda_n(g)}$ and $A_{J}$ commute mutually, we obtain that $\lambda_n(\widetilde C_I^g)$ and $\lambda_n(\widetilde C_J)$ commute mutually and again by injectivity of $\lambda_n$, $\widetilde C_{I}^g$ and $\widetilde C_J$ commute. Therefore we have shown that $\widetilde C_I^g$ and $\widetilde C_J$ are adjacent.
\end{proof}

By contrast with the embedding $\eta_n$ of $A_{B_n}$ in $A_{A_n}$ (Section \ref{S:B}), not all curves in $\mathbb D_{n+2}$ can be obtained as $\mathcal C_I^{\lambda_n(g)}$, for a proper subinterval $I$ of $[n+1]$ and $g\in A_{\widetilde C_n}$. 

\begin{proposition}\label{P:CurveTildeC}
Let $\mathcal C$ be a curve in $\mathbb D_{n+2}$; the following are equivalent:
\begin{itemize}
\item[(i)] There exist a proper subinterval $I$ of $[n+1]$ and $g\in A_{\widetilde C_n}$ such that $\mathcal C=\mathcal C_I^{\lambda_n(g)}$,
\item[(ii)]  $\mathcal C$ does not surround both the first and the last punctures.
\end{itemize}
\end{proposition}

\begin{proof}
(i) $\Longrightarrow$ (ii) Suppose that $\mathcal C=C_{I}^{\lambda_n(g)}$ for some proper subinterval~$I$ of~$[n+1]$ and some $g\in A_{\widetilde C_n}$. The curve~$\mathcal C_{I}$ cannot surround both the first and the last punctures as it is standard; assume for instance that it does not surround the first puncture (the other case is similar). Then as $\lambda_n(g) \in \mathfrak P$, we see that $\mathcal C_{I}^{\lambda_n(g)}$ does not surround the first puncture either. 

(ii) $\Longrightarrow$ (i). Suppose that $\mathcal C$ does not surround both the first and the last punctures. 
We must show that $\mathcal C$ can be transformed into a standard curve by the action of some braid $\beta$ in $\mathfrak P$. 
Suppose for instance that $\mathcal C$ does not surround the first puncture (the other case is similar). By Proposition~\ref{P:Surj}, we know that there exist a \emph{1-pure} braid $\alpha$ and a proper subinterval $I$ of $[n+1]$ such that $\mathcal C^{\alpha}=\mathcal C_I$ is a standard curve surrounding punctures $m, \ldots, m+k$, for some $m\geqslant 2$, $k\geqslant 1$. 

Let $\pi_{\alpha}\in \mathfrak S_{n+2}$ be the permutation associated to $\alpha$. 
Let $j_0=\pi_{\alpha}(n+2)\in \{2,\ldots,n+2\}$. Recall the braid $b_{j_0}$ from Section \ref{S:Correspondence};
note that $\alpha b_{j_0}\in \mathfrak P$. Use Lemma~\ref{L:ActionOfbi}. 
If $j_0<m$ or $j_0>m+k$, then $\mathcal C_I^{b_{j_0}}$ is standard 
so $\beta=\alpha b_{j_0}$ does the job. 
Otherwise $m\leqslant j_0\leqslant m+k$. If $\mathcal C_I^{b_{j_0}}=\mathcal C_{I}^{b_{m+k}}$ is not standard, we have $m+k<n+2$ and $\mathcal C_{I}^{b_{j_0}}\xi'_I=\mathcal C_{[n+2-k,n+1]}$. By Remark \ref{R:Pure}, $\xi'_I$ is $(n+2)$-pure; it is also 1-pure because $m\geqslant 2$. Therefore we can take $\beta=\alpha b_{j_0} \xi'_I$. 
% \begin{figure}
% \center
% \includegraphics{Tau1bis.pdf}
% \caption{(i) Given a proper subinterval $I$ of $[n+1]$ with $m=\min I$ and $k=\#I$,
% and given $m\leqslant j_0\leqslant m+k$, the action of $b_{j_0}$
%  on $\mathcal C_I$ is the same of that of $b_{m+k}$ and results in a non-standard curve $\mathcal C''_I$ whenever $m+k<n+2$. (ii) The braid $\xi'_I$. }
%\label{F:Tau1bis}
%\end{figure}
% 
%Finally, if $m\leqslant j_0\leqslant m+k$, that is the puncture numbered $j_0$ is enclosed by $\mathcal C_I$, then %$\mathcal C^{\alpha(\sigma_{j_0}\ldots \sigma_{n+1})}=
%$\mathcal C_I^{b_{j_0}}=\mathcal C_{I}^{b_{m+k}}$ is not standard as soon as $m+k<n+2$ (see Figure \ref{F:Tau1bis}). However, setting 
%$$\xi'_I=\begin{cases}
%(\sigma_{m+k-1}\ldots \sigma_{m})\ldots (\sigma_n\ldots \sigma_{n-k+1}) & \text{if $m+k<n+2$},\\
%Id & \text{otherwise,}
%\end{cases}$$ 
%we obtain that $\xi'_I\in\mathfrak P$ (since $m\geqslant 2$) and that 
%$\mathcal C^{\alpha b_{j_0} \xi'_I}=\mathcal C_{[n-k+2,n+1]}$ 
%is standard so we can choose ${\beta=\alpha b_{j_0}\xi'_I\in \mathfrak P}$.  
\end{proof}

\begin{definition}
Let $\mathcal K_C$ be the subgraph of $\mathcal{CG}(\mathbb D_{n+2}$ induced by the curves which do not surround both the first and the last punctures of $\mathbb D_{n+2}$. 
\end{definition}

\begin{corollary}\label{C:MapLambda}
The assignment $\widetilde C_I^g\mapsto \mathcal C_{I}^{\lambda_n(g)}$, where $I$ is 
a proper subinterval of $[n+1]$ and $g\in A_{\widetilde C_n}$, 
 defines a graph isomomorphism $\Lambda_n$ 
from $\mathcal C_{parab}(\widetilde C_n)$ to $\mathcal K_C$. In particular, we have $d_{\mathbb D_{n+2}}(\Lambda_n(P),\Lambda_n(P'))\leqslant d_{\widetilde C_n}(P,P')$ for all proper irreducible parabolic subgroups $P$ and $P'$ of~$A_{\widetilde C_n}$. 
\end{corollary}
\begin{proof}
By Proposition \ref{P:WellDefTildeC}, the assignment $\Lambda_n$ is a well-defined injective map. By Proposition \ref{P:CurveTildeC}, the image of this map is $\mathcal K_C$. By Proposition \ref{P:AdjTildeC}, both $\Lambda_n$ and its inverse are graph
homomorphisms. 
\end{proof}

\begin{proposition}\label{P:DenseTildeC}
The subgraph $\mathcal K_C$ is 1-dense in $\mathcal{CG}(\mathbb D_{n+2})$; the graph $\mathcal C_{parab}(\widetilde C_n)$ has infinite diameter. 
\end{proposition}

\begin{proof}
This amounts to show that given a curve $\mathcal C$ surrounding both the first and the last punctures, it is possible to find another curve $c$ disjoint from $\mathcal C$ and such that $c$ does not surround both the first and the last punctures. If the exterior of $\mathcal C$ contains at least 2 punctures, we take any curve $c$ in the exterior of $\mathcal C$. Otherwise, the interior of $\mathcal C$ contains $n+1\geqslant 3$ punctures and 
we can choose in the interior of $\mathcal C$ any curve which does not surround both the first and the last punctures. For the second part we can argue following the same lines as in the proof of the second part of Proposition~\ref{P:DenseTildeA}. 
\end{proof}

To conclude our study we will now show that $\mathcal K_C$ is hyperbolic. In the next two lemmas, we show that $\mathcal K_C$ satisfies the hypothesis of Theorem \ref{T:Vokes}, from which we conclude that $\mathcal K_C$, and hence $\mathcal C_{parab}(\widetilde C_n)$, is hyperbolic. 

\begin{lemma}\label{L:TwistFreeC}
(i) $\mathcal K_C$ is connected. (ii) The natural action of the pure braid group $PA_{A_{n+1}}$ on~$\mathbb D_{n+2}$ induces an action by isometries on $\mathcal K_C$. (iii) No annulus in $\mathbb D_{n+2}$ can be a witness for all vertice of $\mathcal K_C$. 
\end{lemma}
\begin{proof}
The proof is identical to the proof of Lemma \ref{L:TwistFree}. 
\end{proof}

\begin{remark}\rm\label{R:NotQI}
The same argument as in the proof of Proposition \ref{P:NotQI} shows that the embedding of $\mathcal K_C$ in $\mathcal{CG}(\mathbb D_{n+2})$ is not quasi-isometric.
\end{remark}

For the following lemma, we denote by $p_1$ the first puncture and by $p_{n+2}$ the last puncture.

\begin{lemma}\label{L:WitnessC}
Let $X$ be a subsurface of $\mathbb D_{n+2}$. Then $X$ is a witness for $\mathcal K_C$ if and only if one of the following holds. 

\begin{itemize}
\item[(i)] $X=\mathbb D_{n+2}$ or $X=\mathbb D_{n+2}\setminus D$, where $D$ is the interior of an essential curve surrounding $p_1$ and $p_{n+2}$ and no other puncture.  
\item[(ii)] $X$ is the interior of an essential curve surrounding $p_1,p_{n+2}$ and exactly $(n-1)$ other punctures. 
\item[(iii)] $X=X'\setminus D$, where $X'$ is the interior of an essential curve surrounding $p_1,p_{n+2}$ and exactly $(n-1)$ other punctures and $D$ is the interior of an essential curve surrounding $p_1$ and $p_{n+2}$ and no other puncture.  
\end{itemize}
We will say that $X$ is a witness of type (i), (ii) or (iii). Two witnesses for $\mathcal K_C$ are never disjoint. 
\end{lemma}

\begin{proof}
\emph{Mutatis mutandis}, the proof is the same as the proof of Lemma \ref{L:Witness}. 
\end{proof}

The proof of the statements of Theorem \ref{T:TildeA} concerning $\mathcal C_{parab}(\widetilde C_n)$ is now complete: (i) was seen in the introduction, (ii) is the statement of Corollary \ref{C:MapLambda}, (iii) is the statement of Proposition \ref{P:DenseTildeC} and (iv) results from Lemmas \ref{L:TwistFreeC}, \ref{L:WitnessC} and Theorem \ref{T:Vokes}.

\section{Hyperbolic structures on Artin-Tits groups}\label{S:Comments}

In this section we briefly review some of the -known or conjectural- hyperbolic structures on Artin-Tits groups presented in \cite{CalvezWiest} and we show some connections with our results. A generating set $X$ of a group $G$ is a \emph{hyperbolic structure} if the Cayley graph $\Gamma(G,X)$ of $G$ with respect to $X$ is a hyperbolic metric space. 

\subsection{Hyperbolic structures on $A_{B_n}$}

Let $A_{\Gamma}$ be an Artin-Tits group of spherical type and consider the following generating sets for $A_{\Gamma}$. 

\begin{itemize}
\item $X_{NP}(\Gamma)$ is the union of the normalizers of the proper irreducible standard parabolic subgroups of $A_{\Gamma}$. 
\item $X_P(\Gamma)$ is the union of the proper irreducible standard parabolic subgroup of $A_{\Gamma}$ and the cyclic subgroup generated by the square of the element $\Delta_{\Gamma}$ (recall from Section \ref{S:ArtinGroups} that $\Delta_{\Gamma}$ is the lift of the longest element of the corresponding Coxeter group $W_{\Gamma}$).
\item $X_{abs}(\Gamma)$ is the set of absorbable elements as described in \cite{CalvezWiest1}. 
\end{itemize}

By \cite[Theorem 1]{CalvezWiest1}, we know that $X_{abs}(\Gamma)$ is a hyperbolic structure on $A_{\Gamma}$; moreover, this is the only one of the three sets which is known to be a hyperbolic structure for all $\Gamma$. 
Both $X_{NP}(A_n)$ and $X_P(A_n)$ are hyperbolic structures; indeed, $Cay(A_{A_n}, X_{NP}(A_n))$ is quasi-isometric to 
$\mathcal{CG}(\mathbb D_{n+1})$ \cite[Proposition 3.2]{CalvezWiest}, while $Cay(A_{A_n}, X_{P}(A_n))$ is quasi-isometric to~$\mathcal A_{\partial}(\mathbb D_{n+1})$, the graph of arcs in $\mathbb D_{n+1}$ both of whose endpoints lie in the boundary $\partial\mathbb D_{n+1}$ \cite[Proposition 3.4]{CalvezWiest}. 

A classical argument (see \cite[Lemma 2.5, Proposition 4.4]{CalvezWiest}) shows that (except for dihedral Artin-Tits groups), $Cay(A_{\Gamma}, X_{NP}(\Gamma))$ is quasi-isometric to the graph of irreducible parabolic subgroups of $A_{\Gamma}$. Therefore, Theorem \ref{T:B} can be rephrased by saying that $X_{NP}(B_n)$ is a hyperbolic structure on $A_{B_n}$. We shall prove an analogous statement for the generating set $X_P(B_n)$. 
Again, this will be obtained by comparing $Cay(A_{B_n},X_P(B_n))$ and $Cay(A_{A_n},X_P(A_n))$.

Recall from Section \ref{S:B} the monomorphism $\eta_n: A_{B_n}\longrightarrow A_{A_n}$ whose image is the subgroup $\mathfrak P_1$ of 1-pure braids on $(n+1)$ strands. For $1\leqslant i\leqslant n$, let $a_i=\sigma_i\ldots \sigma_1$ and $a_0=Id$; for each $y\in A_{A_n}$, there is a unique~$i\in \{0,\ldots,n\}$ such that $ya_i\in \mathfrak P_1$. For any braid $y\in A_{A_n}$, we denote by $\pi_y$ the associated permutation in $\mathfrak S_{n+1}$. The next lemma resembles Lemma \ref{L:ActionOfai}. 

\begin{lemma}\label{L:Technical}
Let $I$ be a proper subinterval of $[n]$, $m=\min(I)$ and $k=\#I$, so that the circle $\mathcal C_I$ in $\mathbb D_{n+1}$ surrounds the punctures $m$ to $m+k$. 
Let $0\leqslant i_0,j_0\leqslant n$.  
%Let $I'=\{i+1\ |\ i\in I\}$ (whenever $\max(I)<n$). 
Let $g\in A_I$ and suppose that $z=a_{i_0}^{-1}ga_{j_0}$ is \emph{1-pure}. 
\begin{itemize}
\item[(i)] If $i_0+1<m$, then $z=g\in A_I$. 
\item[(ii)] If $i_0+1>m+k$, then $z=sh(g)\in A_{I'}$, where $I'=\{i+1\ |\ i\in I\}$. 
\item[(iii)] If $m\leqslant i_0+1\leqslant m+k$, then $z^{\xi_I}\in A_{\{1,\ldots,k\}}$.
\end{itemize}
Here, $\xi_I\in \mathfrak P_1$ is the braid defined in Section \ref{S:Correspondence} which satisfies $A_I^{a_{m-1}\xi_I}=A_{[1,k]}$ and $sh$ denotes the shift homomorphism $\sigma_i\mapsto \sigma_{i+1}$. 
\end{lemma}

\begin{proof}
First observe that as $z$ is 1-pure, we must have $\pi_g(i_0+1)=j_0+1$. Moreover, as $g\in A_I$, $\pi_g(i)=i$ for all $i\in \{1,\ldots,m-1\}\cup\{m+k+1,\ldots,n+1\}$. In particular, if $i_0+1<m$ or $i_0+1>m+k$, we must have $i_0+1=j_0+1$, whence $i_0=j_0$.

(i) Suppose that $i_0+1<m$; as we have just seen, $z=a_{i_0}^{-1}ga_{i_0}$. But $a_{i_0}$ commutes with all letters $\sigma_i, i\in I$ whence $z=g$. 

(ii) Suppose that $i_0+1>m+k$; again $z=a_{i_0}^{-1}ga_{i_0}$. We have, for all $i\in I$, 
\begin{eqnarray*}
a_{i_0}^{-1}\sigma_ia_{i_0} & = &( \sigma_1^{-1}\ldots \sigma_{i_0}^{-1})\sigma_i(\sigma_{i_0}\ldots \sigma_1)\\
 & = &\sigma_1^{-1}\ldots\sigma_i^{-1}(\sigma_{i+1}^{-1}\sigma_i\sigma_{i+1})\sigma_i\ldots \sigma_1\\
  & = & \sigma_1^{-1}\ldots\sigma_i^{-1} (\sigma_i \sigma_{i+1}\sigma_i^{-1})\sigma_i\ldots \sigma_1\\
   & = & \sigma_{i+1},
\end{eqnarray*}
and the claim follows.\\
(iii) Suppose that $m\leqslant i_0+1\leqslant m+k$; then also $m\leqslant j_0+1\leqslant m+k$. 
We have 
\begin{eqnarray*}
z & = & a_{i_0}^{-1}ga_{j_0}\\
 & =  & (\sigma_1^{-1}\ldots\sigma_{m-1}^{-1})(\sigma_m^{-1}\ldots \sigma_{i_0}^{-1}g\sigma_{j_0}\ldots\sigma_{m})(\sigma_{m-1}\ldots \sigma_1)\\
  & = & (\sigma_{1}^{-1}\ldots \sigma_{m-1}^{-1}) g' (\sigma_{m-1}\ldots \sigma_1),\\
 \end{eqnarray*}
 where $g'=(\sigma_m^{-1}\ldots \sigma_{i_0}^{-1})g(\sigma_{j_0}\ldots\sigma_{m})$ (note that the first and third factor may be trivial if $i_0+1=m$ or $j_0+1=m$). Since all crossings involved are inner to $\mathcal C_I$, $g'\in A_I$. 
 We deduce that $z=g'^{a_{m-1}}\in A_I^{a_{m-1}}$ and then by definition of $\xi_I$, $z^{\xi_I}\in A_{[1,k]}$ as claimed. 
\end{proof}

\begin{proposition}\label{P:XPB}
Let $n\geqslant 3$.
The monomorphism $\eta_n: A_{B_n}\longrightarrow A_{A_n}$ induces a quasi-isometry between $Cay(A_{B_n}, X_P(A_{B_n}))$ and $Cay(A_{A_n}, X_P(A_{A_n}))$.
\end{proposition}

\begin{proof}
Throughout the proof,  the notation $||x||_{X_P(B_n)}$ means the word length of $x\in A_{B_n}$ with respect to the generating set $X_P(B_n)$. 

 We know by Lemma \ref{L:Comput}(i) that given a proper subinterval $I$ of $[n]$, $\eta_n(B_I)=A_I\cap \mathfrak P_1$; therefore, for $x,y\in A_{B_n}$, $x^{-1}y\in B_I$ is equivalent to $\eta_n(x)^{-1}\eta(y)\in A_I$. Similarly, $x^{-1}y=\Delta_{B_n}^{2k}$ for some $k\in \mathbb Z$ is equivalent to $\eta_n(x)^{-1}\eta_n(y)=\eta_n(\Delta_{B_n}^{2k})=\Delta_{A_n}^{4k}$. Therefore,~$\eta_n$ induces a 1-Lipschitz map from $Cay(A_{B_n}, X_P(B_n))$ to $Cay(A_{A_n}, X_P(A_n))$. 
 
We define a map $\psi_n: A_{A_n} \longrightarrow A_{B_n}$ in the following way. Given $y\in A_{A_n}$, let $i\in \{0,...,n\}$ (it is unique!) be such that $ya_i\in  \mathfrak P_1=Im(\eta_n)$, and define $\psi(y) =\eta_n^{-1}(ya_i)$. Let us see that $\psi_n$ is a quasi-inverse for $\eta_n$. Indeed, we have by construction, for $x\in A_{B_n}$, $\psi_n\circ \eta_n (x)=x$. 
Conversely, for $y\in A_{A_n}$, $y$ and $\eta_n\circ\psi_n(y)$ differ by $a_i$, for some $0\leqslant i\leqslant n$; however $a_i$ can be written as a product of at most 2 elements in $X_P(A_{A_n})$, whence the distance between $y$ and $\eta_n\circ\psi_n(y)$ in $Cay(A_{A_n},X_P(A_n))$ is at most 2. 

We show finally that the map $\psi_n$ is Lipschitz. Let $y,y'\in A_{A_n}$ be adjacent in $Cay(A_{A_n},X_P(A_n))$; write $g=y^{-1}y'$. This means that $g\in X_P(A_n)$, that is $g=\Delta_{A_n}^{2k}$ for $k\in \mathbb Z$, or $g\in A_I$ for some proper subinterval $I$ of $[n]$. We have unique $i_0, j_0$ so that $ya_{i_0}\in \mathfrak P_1$ and $y'a_{j_0}\in \mathfrak P_1$. Then also $a_{i_0}^{-1}g a_{j_0}$ is 1-pure. Let $x=\eta_n^{-1}(ya_{i_0})$ and $x'=\eta_n^{-1}(y'a_{j_0})$; we must estimate $||x^{-1}x'||_{X_P(B_n)}$. 

If $g=\Delta_{A_n}^{2k}$, $g$ is pure. As $a_{i_0}^{-1}g a_{j_0}$ is 1-pure, we must have $i_0+1=\pi_g(i_0+1)=j_0+1$, whence $i_0=j_0$.  We deduce, as $g$ is central, $\eta_n(x^{-1}x')=a_{i_0}^{-1}g a_{j_0}=g=\Delta_{A_n}^{2k}$ and $x^{-1}x'=\Delta_{B_n}^k$. In summary, $||x^{-1}x'||_{X_P(B_n)}\leqslant ||\Delta_{B_n}||_{X_P(A_{B_n})}+1$ (and this bound is 1 if $k$ is even). 

If $g\in A_I$ for some proper subinterval $I$ of $[n]$, then Lemma \ref{L:Technical} says that there is some proper subinterval $J$ of $[n]$ so that either $a_{i_0}^{-1}ga_{j_0}\in A_J$ or~$(a_{i_0}^{-1}ga_{j_0})^{\xi_I}\in A_J$. Pulling back to $A_{B_n}$ this assertion, we see 
using Lemma \ref{L:Comput}(i) that either $x^{-1}x'\in B_J$ or $(x^{-1}x')^{\eta_n^{-1}(\xi_I)}\in B_J$. 
It follows that $||x^{-1}x'||_{X_{P}(B_n)}\leqslant 1+2||\eta_n^{-1}(\xi_I)||_{X_P(B_n)}.$ This is uniformly bounded as the set of proper connected subintervals of~$[n]$ is finite. 
\end{proof}

We can now complete the proof of Theorem \ref{T:XP}. By \cite[Proposition 3.4]{CalvezWiest}, $X_P(A_n)$ is a hyperbolic structure on $A_{A_n}$, that is $Cay(A_{A_n},X_P(A_n))$ is hyperbolic. By Proposition \ref{P:XPB}, $Cay(A_{B_n},X_P(B_n))$ is hyperbolic as well, that is $X_P(B_n)$ is a hyperbolic structure on $A_{B_n}$. 
The second part of the theorem follows immediately from Proposition \ref{P:XPB}, Theorem \ref{T:B} and the corresponding fact for~$A_{A_n}$ \cite[Proposition~4.19]{CalvezWiest}.

\subsection{Non-spherical type}
Even if $A_{\Gamma}$ is not of spherical type, we can extend the definition of $X_{NP}(\Gamma)$ and $X_{P}(\Gamma)$, just dropping the powers of~$\Delta_{\Gamma}$ in the definition of $X_P(\Gamma)$. 
We will see that $X_{NP}(\widetilde  {X_n})$ is a hyperbolic structure on $A_{\widetilde Z_n}$, for $Z=A$ or $C$. 
We leave open whether $X_P(\widetilde Z_n)$ is a hyperbolic structure on~$A_{\widetilde Z_n}$. 

The proof rests on a technical result --simultaneous standardization of adjacent proper irreducible parabolic subgroups-- which generalizes \cite[Proposition 4.4]{LeeLee} and \cite[Section 11]{CGGMW} and could be interesting on its own. The result is split into the next two propositions. 

\begin{proposition}\label{P:EdgeOrbitTildeA}
Let $P,Q$ be adjacent proper irreducible parabolic subgroups of $A_{\widetilde A_n}$. Then there exists $\mathfrak s\in A_{\widetilde A_n}$ so that $P^{\mathfrak s}$ and $Q^{\mathfrak s}$ are standard. 
\end{proposition}

\begin{proof}
By Proposition \ref{P:AdjTildeA}, $\theta_n(P)$ and $\theta_n(Q)$ are adjacent in $\mathcal C_{parab}(B_{n+1})$. 
As $A_{B_{n+1}}$ is of spherical type, by \cite[Section 11]{CGGMW}, there exists $\zeta\in A_{B_{n+1}}$ so that $\theta_n(P)^{\zeta}$ and $\theta_n(Q)^{\zeta}$ are standard parabolic subgroups of $A_{B_{n+1}}$. 
By Proposition \ref{P:Image}, $\theta_n(P)$, $\theta_n(Q)$ and their respective standard conjugates $\theta_n(P)^{\zeta}$ and $\theta_n(Q)^{\zeta}$ are braid subgroups of $A_{B_{n+1}}$, that is 
$\theta_n(P)^{\zeta}=B_I$ and $\theta_n(Q)^{\zeta}=B_J$ for some proper subintervals~$I$ and $J$ of $[n+1]$ not containing 1.
Let $I'=\{i-1\ |\ i\in I\}$ and $J'=\{j-1\ |\ j\in J\}$, so that $\theta_n(P)^{\zeta}=\theta_n(\widetilde A_{I'})$ and $\theta_n(Q)^{\zeta}=\theta_n(\widetilde A_{J'})$. 

Suppose first that $\zeta=\theta_n(\mathfrak s)$ for some $\mathfrak s\in A_{\widetilde A_n}$. Then we have $\theta_n(P^{\mathfrak s})=\theta_n(P)^{\zeta}=\theta_n(\widetilde A_{I'})$ and $\theta_n(Q^{\mathfrak s})=\theta_n(Q)^{\zeta}=\theta_n(\widetilde A_{J'})$. We deduce from the injectivity of $\theta_n$ that $P^{\mathfrak s}=\widetilde A_{I'}$ and $Q^{\mathfrak s}=\widetilde A_{J'}$ are both standard, showing our claim. 

If on the contrary $\zeta$ is not in the image of $\theta_n$, in view of Proposition \ref{P:SemiDirect}(iii) we can write $\zeta=\zeta_0\rho^r$, where $\zeta_0=\theta_n(\mathfrak s)$ for some $\mathfrak s\in A_{\widetilde A_n}$ and 
 $r\in \mathbb Z$.
Using Proposition \ref{P:SemiDirect}(ii), we then have 
$$\theta_n(P^{\mathfrak s})=\theta_n(P)^{\zeta_0}=\theta_n(P)^{\zeta\rho^{-r}}=\theta_n (\widetilde A_{I'}) ^{\rho^{-r}}=\theta_n(\widetilde A_{I''})$$ and 
$$\theta_n(Q^{\mathfrak s})=\theta_n(Q)^{\zeta_0}=\theta_n(Q)^{\zeta\rho^{-r}}=\theta_n (\widetilde A_{J'}) ^{\rho^{-r}}=\theta_n(\widetilde A_{J''}),$$
where $I''=\{i'-r\ |\  i'\in I'\}$, $J''=\{j'-r\ |\  j'\in J'\}$ and the indices are taken modulo $n+1$. 
We obtain that $P^{\mathfrak s}=\widetilde A_{I''}$ and $Q^{\mathfrak s}=\widetilde A_{J''}$ are both standard, as needed. 
\end{proof}

\begin{proposition}\label{P:EdgeOrbitTildeC}
Let $P,Q$ be adjacent proper irreducible parabolic subgroups of $A_{\widetilde C_n}$. Then there exists $\mathfrak s\in A_{\widetilde C_n}$ so that $P^{\mathfrak s}$ and $Q^{\mathfrak s}$ are standard. 
\end{proposition}

\begin{proof}
Recall the graph isomorphism $\Lambda_n$ from Corollary \ref{C:MapLambda}: 
$\mathcal C_1=\Lambda_n(P)$ and $\mathcal C_2=\Lambda_n(Q)$ are 
disjoint curves in $\mathbb D_{n+2}$ which do not surround both the first and the last punctures. 
Under this isomorphism, standard parabolic subgroups of $A_{\widetilde C_n}$ are in correspondence with standard curves of~$\mathbb D_{n+2}$. So our claim will follow from proving 
that there exists $\zeta_0$ in $\mathfrak P=\lambda_n(A_{\widetilde C_n})$ 
so that both~$\mathcal C_1^{\zeta_0}$ and ${\mathcal C_2}^{\zeta_0}$ are standard; in this way, setting $\mathfrak s=\lambda_n^{-1}(\zeta_0)$ will prove the claim. 

Recall the graph isomorphism $\mathfrak H_{n+1}$ from Corollary \ref{C:MapH} which identifies $\mathcal{CG}(\mathbb D_{n+2})$ and $\mathcal C_{parab}(B_{n+1})$. Because $A_{B_{n+1}}$ is of spherical type, simultaneous standardization in $A_{B_{n+1}}$ (see \cite[Section 11]{CGGMW}) yields 
$g\in A_{B_{n+1}}$ such that $\mathfrak H_{n+1}^{-1}(\mathcal C_1)^g$ and $\mathfrak H_{n+1}^{-1}(\mathcal C_2)^g$ are 
standard parabolic subgroups of $A_{B_{n+1}}$. Setting $\alpha=\eta_{n+1}(g)$, we obtain that $\mathcal C_1^{\alpha}$ and $\mathcal C_2^{\alpha}$ are disjoint standard curves, say $\mathcal C_I$ and $\mathcal C_J$ for some proper subintervals $I,J$ of $[n+1]$. Note that $\alpha$ is already 1-pure.

If $\alpha\in \mathfrak P$, we are done by setting $\zeta_0=\alpha$. Otherwise, as in the proof of Proposition \ref{P:AdjTildeC}, let $\pi_{\alpha}$ be the permutation in $\mathfrak S_{n+2}$ associated to $\alpha$, 
$j_0=\pi_{\alpha}(n+2)\in \{2,\ldots, n-1\}$ and $b_{j_0}=\sigma_{j_0}\ldots \sigma_{n+1}$. Then $\alpha b_{j_0}\in \mathfrak P$. 

%The action of $b_{j_0}$ on $\mathcal C_I$ and $\mathcal C_J$ is described in the proof of Proposition \ref{P:AdjTildeC}. 
If both $\mathcal C_I^{b_{j_0}}$ and $\mathcal C_J^{b_{j_0}}$ are standard again, then we can choose $\zeta_0=\alpha b_{j_0}$. Suppose then that~$\mathcal C_I^{b_{j_0}}$ is not standard; then (Lemma \ref{L:ActionOfbi}(iii)) $\mathcal C_I$ must surround the puncture $j_0$ --hence it cannot surround the first puncture. 
If $\mathcal C_J^{b_{j_0}}$ is standard,  then the braid $\xi'_I$ from Lemma \ref{L:ActionOfbi} satisfies that $\mathcal C_e^{\alpha b_{j_0}\xi'_I}$ is standard for $e=1,2$;  moreover, $\xi_i'\in\mathfrak P$ (Remark \ref{R:Pure} and because $\mathcal C_I$ does not surround the first puncture). Therefore we can choose $\zeta_0=\alpha b_{j_0}\xi'_I\in \mathfrak P$. 
If $\mathcal C_J^{b_{j_0}}$ is not standard, then both $\mathcal C_I, \mathcal C_J$ have to be nested. 
Assume that $\mathcal C_J$ is in the interior of $\mathcal C_I$. Let 
$m_I=\min I$, $m_J=\min J$, $k_I=\#I$ and $k_J=\#J$. We define 
%$$\xi= \xi_J' (\sigma_{m_I+k_I-k_J-1}\ldots \sigma_{m_I})\ldots (\sigma_{n-k_J}\ldots \sigma_{n+1-k_I}).$$
$$\xi=\xi'_{J} \xi_{m_I,m_I+k_I-k_J-1,n+1-k_J}.$$
We note that $\xi$ does not use any letter $\sigma_1$ nor $\sigma_{n+1}$, so that $\xi\in \mathfrak P$. Moreover, $\mathcal C_I^{b_{j_0}\xi}$ and $\mathcal C_J^{b j_0\xi}$ are standard. This concludes the proof 
choosing $\zeta_0=\alpha b \xi$. 
\end{proof}

\begin{corollary}
Assume that $\Gamma$ is either $\widetilde A_n$ or $\widetilde C_n$ ($n\geqslant 2$). Then $\mathcal C_{parab}(\Gamma)$ is quasi-isometric to $Cay(A_{\Gamma}, X_{NP}(\Gamma))$ and $X_{NP}(\Gamma)$ is a hyperbolic structure on $A_{\Gamma}$.
\end{corollary}

\begin{proof}
By Theorem \ref{T:TildeA}(i), $\mathcal C_{parab}(\Gamma)$ is connected; the proper irreducible standard parabolic subgroups form a finite set representing  all the orbits of vertices under the natural action of $A_{\Gamma}$ and by Propositions \ref{P:EdgeOrbitTildeA} and \ref{P:EdgeOrbitTildeC}, the finite set consisting of all edges bounded by two standard parabolic subgroups
is a set of representatives of the orbits of edges under the action of  $A_{\Gamma}$. By \cite[Lemma~2.5]{CalvezWiest}, this implies that $\mathcal C_{parab}(\Gamma)$ and $Cay(A_{\Gamma}, X_{NP}(\Gamma))$ 
are  quasi-isometric. By Theorem \ref{T:TildeA}(iv), $X_{NP}(\Gamma)$ is a hyperbolic structure on $A_{\Gamma}$. 
\end{proof}

Now, we comment on the last generating set considered above: $X_{abs}(\Gamma)$ is the set of \emph{absorbable}e elements. The concept of an absorbable element was introduced in the context of a finite-type Garside group \cite{CalvezWiest1}. The only irreducible Artin-Tits groups which are finite-type Garside groups are those of spherical type \cite{BS}. However, the euclidean Artin-Tits groups under study in this paper possess a Garside structure, of \emph{infinite} type --see \cite{Digne1,Digne2}. Moreover, McCammond and his collaborators have generalized this situation: although no other euclidean Artin-Tits group than $A_{\widetilde A_n}$ and $A_{\widetilde C_n}$ (and $A_{\widetilde G_2}$) has such a Garside structure, they show that every euclidean Artin-Tits group embeds in a group with an infinite-type Garside structure, the so-called \emph{crystallographic Garside group} \cite{McCS}. 
This raises the following question. 
\begin{problem}\rm 
Is it possible to generalize to an infinite-type Garside structure the construction of the absorbable elements and the additional length graph from \cite{CalvezWiest1}? Does this yield a hyperbolic graph? If yes, can this be used to show that euclidean Artin-Tits groups are acylindrically hyperbolic, as in \cite{CalvezWiest2}? \end{problem} 

{\bf{Acknowledgements}} 
This work builds on an early draft by the first author; he thanks Juan Gonz\'alez-Meneses for pointing a mistake in that preliminary version. Most of the work was done during a visit of the second author at Universidad de Valpara\'iso supported by FONDECYT Regular 1180335. The first author is supported by FONDECYT Regular 1180335, MTM2016-76453-C2-1-P and FEDER. Heartfelt thanks to Bert Wiest for bringing the reference \cite{vokes} to our attention.

\end{document}